\documentclass[12pt]{amsart}

\usepackage{amsmath, amsthm, amssymb, amsfonts} 
\usepackage{mathtools} 
\usepackage{bm} 
\usepackage{graphicx} 
\usepackage{hyperref} 
\usepackage{cleveref} 
\usepackage{geometry} 
\usepackage{enumitem} 
\usepackage{booktabs} 
\usepackage{siunitx} 
\usepackage{microtype} 
\usepackage{xcolor}
\usepackage{a4wide}

\allowdisplaybreaks[4]

\theoremstyle{plain}
\newtheorem{theorem}{Theorem}[section]
\newtheorem{lemma}[theorem]{Lemma}
\newtheorem{proposition}[theorem]{Proposition}

\numberwithin{equation}{section}

\theoremstyle{definition}

\newtheorem{remark}[theorem]{Remark}

\newcommand{\ep}{\varepsilon}

\title[Multiple  helical vortices]{Dynamics and leapfrogging phenomena of multiple helical vortices for  3D incompressible Euler equations}

\author{Daomin Cao, Junhong Fan, Guolin Qin, Jie Wan}
\address{State Key Laboratory of Mathematical Sciences, Academy of Mathematics and Systems Science, Chinese Academy of Sciences, Beijing 100190, P.R. China and University of Chinese Academy of Sciences, Beijing 100049,  P.R. China}
\email{dmcao@amt.ac.cn}

\address{Institute of Applied Mathematics, AMSS, Chinese Academy of Sciences, Beijing 100190, and University of Chinese Academy of Sciences, Beijing 100049,  P.R. China}
\email{fanjunhong@amss.ac.cn}

\address{State Key Laboratory of Mathematical Sciences, Academy of Mathematics and Systems Science, Chinese Academy of Sciences, Beijing 100190, P.R. China}
\email{qinguolin18@mails.ucas.ac.cn}

\address{School of Mathematics and Statistics, Beijing Institute of Technology, Beijing 100081, P.R. China}
\email{wanjie@bit.edu.cn}

\begin{document}

\begin{abstract}
	In this paper, we investigate the time evolution of helical vortices without swirl for the incompressible Euler equations in  $\mathbb R^3$ under general  initial assumptions. Assume the initial helical vorticity is sharply concentrated in $N$ distinct  $\ep$-neighborhoods, whose mutual distances vanish as   $O(1/|\ln \ep|)$, and each vortex core possesses vorticity mass  of order $1/|\ln \ep|^{1+b}$ for an arbitrary fixed  $b\in\mathbb R$. We prove that as  $\ep\to 0$, the motion of these helical vortices converges uniformly to a dynamical system derived herein over a time interval of order  $1/|\ln\varepsilon|^{1-b}$. In the particular case  $b=-1$, our results establish the evolution counterpart for interacting vortex helices constructed in [I. Guerra, M. Musso, Ann. Inst. H. Poincar\'e C Anal. Non Lin\'aire, 2025]. Notably, for two interacting helical vortices with initial mutual distance  $ \rho_0/|\ln \ep|$, by choosing $\rho_0$ sufficiently small, our analysis extends to timescales covering multiple periods. This result provides the first mathematical justification for the numerically observed phenomenon termed ``leapfrogging of Kelvin waves" reported in [N. Hietala et al., Phys. Rev. Fluids, 2016].
\end{abstract}

 \maketitle{\small{\bf Keywords:}  Incompressible Euler equations, helical vortex, dynamical system, leapfrogging phenomena.   \\	
	
	}

\section{Introduction}
\label{sec:introduction}

We consider the three-dimensional incompressible Euler equations
\begin{equation}\label{Euler}
    \left\{
\begin{array}{ll}
\partial_t u+(u\cdot\nabla)u=-\nabla P  & \text{in} \quad \mathbb R^3 \times \mathbb{R}_+, \\
\mathrm{div}(u)=0 & \text{in} \quad \mathbb R^3 \times \mathbb{R}_+,\\
u(\cdot,0)=u_0  &\text{in} \quad \mathbb R^3,
\end{array}
\right.
\end{equation}
where   $u:\mathbb R^3 \times \mathbb{R}_+\to \mathbb{R}^3$ denotes the velocity of the fluid and $P$ the scalar pressure. For a solution  $u $ of \eqref{Euler}, the corresponding vorticity field of the fluid is defined by $$\bm{\omega} :=\nabla\times u.$$

Taking curl of the first equation in \eqref{Euler}, we obtain the following Euler equations of vorticity form:
\begin{align}\label{Eulerw}
	\begin{cases}
		\partial_t \bm{\omega} +(u\cdot\nabla)\bm{\omega} =(\bm{\omega} \cdot\nabla)u,
		\ \ & \text{in} \quad \mathbb R^3 \times \mathbb{R}_+,\\
		\bm{\omega} (\cdot, 0)=\nabla\times u_0(\cdot), \ \ & \text{in} \quad \mathbb R^3.
	\end{cases}
\end{align}
The velocity field $u$ in \eqref{Eulerw} can be recovered from $\bm{\omega}$ by the so-called Biot-Savart law:
\begin{equation}\label{Biot-Savart}
	u=\nabla\times(-\Delta)^{-1}\bm{\omega}.
\end{equation}

\subsection{Backgrounds}

In this paper, we are interested in the dynamics of flow, referred to as vortex filaments,  when the   vorticity is supported in a very small neighborhood of a smooth curve. The study of such problem can be traced back to Helmholtz's classical paper in 1858 \cite{Hel}.  Helmholtz \cite{Hel} first found that vorticity fields of \eqref{Eulerw} with axial symmetry (called the vortex rings)  supported in a toroidal region with small cross-section near a circular curve have an approximately steady form and travel with a large constant velocity along the axis of the ring. For general curves, the first derivation of the motion law of curves dates back to the work of Da Rios  in 1906 \cite{DR} and  Levi-Civita  in 1908 \cite{LC,LC2} with the help of potential theory. Consider a vortex tube with a small cross-section of radius $\ep$ and a fixed circulation $ \kappa $, uniformly distributed around an evolving smooth curve $\Gamma (t) $. By using the Biot-Savart law and calculating the leading order of instantaneous velocity of the curve, \cite{DR} indicates that the  curve $\Gamma (t) $ evolves by its \textit{binormal curvature flow}: let $\Gamma(t)$ be parameterized by $\chi(s,t)$ with $s$ the arc length parameter, then the motion of $\chi$ would asymptotically  obey a law of the form
\begin{equation}\label{BNflow}
	\partial_t \chi=\frac{\kappa}{4\pi}|\ln\varepsilon|(\partial_s\chi\times\partial_{ss}\chi),
\end{equation}
where $\kappa$ is the circulation of the flow.

A rigorous mathematical justification of the relation between vorticity field of \eqref{Eulerw} and the curve \eqref{BNflow} remains a subject of great concerned, known as \emph{vortex filament conjecture}. We refer to   \cite{FV, JS} and references therein for comprehensive introduction and some recent  development on this conjecture.  Numerous partial results to the vortex filament conjecture, concerning long time location of vorticity are available for specific  curves. We refer to, e.g. \cite{BM, CJY, DL, MP, Tur87} for the case of straight vortex filaments and related results. Time evolution and dynamics for vortex rings without swirl were studied in, e.g. \cite{BCM, BCM1, BCM2, BCM3, BM2, CFLQ, DHLM}. The third particular curve satisfying \eqref{BNflow},  corresponding to the traveling-rotating helix, has recently garnered widespread attentions. By solving a suitably chosen elliptic PDE, D\'avia et al. \cite{DDMW2} constructed  traveling-rotating vortex helices that do not change their shape as time evolves. We refer to \cite{CW2, GM} and references therein for more literature on the construction of helical vortices with specific initial data. Subsequently, Donati et al. \cite{DLM} first studied the evolution of helical vortices with general vorticity, where it is  proved that any helical vorticity solution initially highly concentrated around helices of pairwise distinct radii remains concentrated in a small neighborhood of filaments for $t\in[0, T/|\ln\ep|)$.  This result was later improved to longer time interval $[0, T)$ in \cite{GZ} for single vortex helix by using a different method. We would like to mention the recent breakthroughs on the evolution of vortex rings in viscous fluids \cite{BCM, GS2}.

\smallskip

In this paper, we are interested in the interaction of multiple vortex filaments. It has been  shown in \cite{BCM1} that the interactions of vortex rings, which are initially separated from each other by distances with a positive lower bound independent  of  the sizes of cross-section, can be neglected compared to their self-interaction. Therefore, each vortex ring dynamics follows the same law as  those of isolated vortex rings.  Similar conclusion was then obtained for helical vortices in \cite{DLM}. However, a well-known  phenomenon, termed ``leapfrogging of vortex rings", was predicted in \cite{Hel}, which states that  in particular conditions, two coaxial vortex rings can successively pass through each other in a periodic motion. Rigorous verification of this phenomenon for vortex rings for the Euler equations remains a long standing open problem until the work by D$\acute{\text{a}}$vila et al. \cite{DDMW} in a constructive way via selecting special initial vorticity. For a general class of initial data corresponding to multiple concentrated vortex rings, localization estimates and convergence to   dynamical systems accounting for the leapfrogging phenomena was studied in \cite{BCM3, DHLM} respectively under different scaling regimes. In particular, for the case of two vortex rings with sufficiently large radius, \cite{BCM3} showed that leapfrogging phenomenon occurs at least for several overtakings. For constructions of leapfrogging vortex patches for the 2D Euler equations  and corresponding vortex ring configurations for Gross-Pitaevskii systems, we refer to \cite{HHM} and \cite{JS2}, respectively.

Leapfrogging phenomena for helical vortices, referred to as ``leapfrogging of Kelvin waves" or ``leapfrogging of helical vortex filaments", were also observed  in many numerical and experimental studies, see e.g. \cite{CL, DSR, HHSB} and references therein. However, rigorous mathematical verification of such phenomena  still remains as an outstanding open problem.

In this paper, we will investigate the interaction of multiple helical vortices  without  swirl under suitable setting, establishing  strong localization estimates and convergence to a dynamical system for a positive time independent of $\ep$ as $\ep\to0$. In particular, for two helical vortices with sufficiently small initial mutual distance, we improve our result to reach longer time interval to cover  any given number of  periods of the motion for the derived dynamical system. This provides the first justification of leapfrogging phenomena for helical vortices.

  \subsection{Main result} In order to introduce our main results, we need some definitions.  For $x\in\mathbb R^2$ and $r>0$, we shall denote $B(x,r)$ to the disk centered at $x$ with radius $r$. For  fix pitch $h>0$,  we define the following operators for all $\theta\in \mathbb{R}$:
\begin{equation*}
    R_{\theta}=\begin{pmatrix}
        \cos\theta & -\sin\theta&0\\\sin\theta&\cos\theta&0\\0&0&1
    \end{pmatrix}
    \quad and \quad S_{\theta,h}x=R_\theta x+h \begin{pmatrix}
        0\\0\\\theta
    \end{pmatrix},\quad x\in \mathbb R^3.
\end{equation*}
 We say that $(u,P)$ is a helical solution to \eqref{Euler}   if $u$ is a helical vector field and $P$ is a helical function in the following sense:
\begin{equation*}
    u(S_{\theta,h}x)=R_\theta u(x),\quad P(S_{\theta,h}x)=P(x),\quad \forall \theta \in \mathbb{R}.
\end{equation*}
A vector field $u:\mathbb R^3\mapsto\mathbb R^3$ is said to be helical without swirl if  $u $ is helical and satisfies
\begin{equation}\label{noswirl}
    u(x)\cdot \xi(x)\equiv 0,  \quad \xi(x)=\begin{pmatrix}
        -x_2\\x_1\\h
    \end{pmatrix}.
\end{equation}
The key observation is that the without swirl condition implies that the vorticity is parallel to $\xi$ which allows us to define a scalar quantity $\omega$:
\begin{equation}
    \nabla\times u(x,t)=\frac{1}{h}\omega(\tilde{R}_{-\frac{x_3}{h}}(x_1,x_2),t) \xi(x),\quad \text{where}\quad \tilde{R}_{\theta}=\begin{pmatrix}
        \cos\theta&-\sin\theta\\\sin\theta&\cos\theta
    \end{pmatrix}.
\end{equation}
Then, it was proved in e.g. \cite{Du, ET} that the Euler equation \eqref{Eulerw}-\eqref{Biot-Savart} can be reduced to a two-dimensional system for the vorticity posed on the plane (we consider the system after time rescaling by $|\ln\varepsilon|^{1-b}$, and $b\in\mathbb{R}$):
\begin{equation}\label{HEuler}
    \left\{
\begin{array}{ll}
\partial_t\omega+\frac{1}{|\ln\varepsilon|^{1-b}}v\cdot\nabla\omega=0  & \text{in} \quad \mathbb R^2 \times \mathbb{R}_+, \\
v=\nabla^\perp \Psi, \ \ \  \mathrm{div}(K\nabla\Psi)=\omega& \text{in} \quad \mathbb R^2 \times \mathbb{R}_+,\\
\omega(\cdot,0)=\omega_0& \text{in} \quad \mathbb R^2,
\end{array}
\right.
\end{equation}
where the convention $(a,b)^\perp=(-b,a)$ and $K$ is a symmetric positive-definite matrix given by
\begin{equation}\label{def K}
    K(x)=\frac{1}{x_1^2+x_2^2+h^2}\begin{pmatrix}
        h^2+x_2^2&-x_1x_2\\-x_1x_2&h^2+x_1^2
    \end{pmatrix}.
\end{equation}
Global well-posedness for the Cauchy problem \eqref{HEuler} are well-understood after seminal works \cite{Abi, BLN, Du, ET, GZ0, JLN}. Nevertheless, similar as the 2D Euler equations, dynamics and long time behavior for solutions of  \eqref{HEuler} remain  unclear. In this paper, we will study dynamics of solutions to \eqref{HEuler} with concentrated initial vorticity.

Inspired by the setting of initial data for  vortex rings considered in \cite{BCM, BCM1, BCM2, BCM3, BM2},  we study the time evolution of an incompressible non-viscous fluid, in the case of helical symmetry without swirl, when the initial vorticity is supported and sharply concentrated in $N$ helices of radius of leading term $r_0$, thickness of order $ \varepsilon $, vorticity mass (in the cross section) of order $1/|\ln\varepsilon|^{1+b}$ with arbitrarily fixed $b\in\mathbb R$, and distance from each other of order $1/|\ln \varepsilon|$. We remark that our assumption on vorticity mass is extremely general due to the arbitrariness of the parameter $b$. In particular, taking $b=-1$, our results can be viewed as the evolution counterpart for interacting vortex helices as considered in \cite{GM}.

To be precise, letting $x_0=(r_0,0)$ be a fixed point with $r_0\geq 0$ and $\varepsilon_0$ is small enough (will depend on $b$),  we assume that the initial data satisfies the following assumptions:
\begin{equation}\label{1}
	\omega_{\varepsilon}(x,0)=\sum^N_{i=1}\omega_{i,\varepsilon}(x,0)\quad \forall\ \varepsilon\in(0,\varepsilon_0),
\end{equation}
where each $\omega_i(x,0)$ has a definite sign, and satisfies the following assumptions
\begin{equation}\label{2}
	|\omega_{i,\varepsilon}(x,0)|\leq\frac{M}{\varepsilon^2},\quad \text{for some}\ \ M>0,
\end{equation}
\begin{equation}\label{3}
	\int\omega_{i,\varepsilon}(x,0)dx=\gamma_i=\frac{a_i}{|\ln\varepsilon|^{1+b}},\quad \text{for some}\ \ a_i\neq 0,
\end{equation}
and
\begin{equation}\label{4}
	\Lambda_{i,\varepsilon}(0):=\mathrm{supp} \ \omega_{i,\varepsilon}(\cdot,0)\subset B(x_0+{P_i^0}/{|\ln \varepsilon|},\varepsilon),
\end{equation}
where $P_i^0\in \mathbb R^2$ ($i=1,\ldots, N$) are given points such that $P_i^0\not=P_j^0$ for $i,j\in\{ 1,\ldots, N \}$, $i\not=j$.

Then the corresponding solution to \eqref{HEuler} can be written as:
\begin{equation}\label{5}
	\omega_{\varepsilon}(x,t)=\sum^N_{i=1}\omega_{i,\varepsilon}(x,t)\quad \forall \varepsilon\in(0,\varepsilon_0).
\end{equation}

To give the limiting dynamic system, let us first introduce some notations. Define
\begin{equation}\label{Tau}
	 g(s)=\frac{1}{2(h\sqrt{s+h^2}+h^2)},\quad \tau(s)=\exp\left(\int^s_0g(z)dz\right)\geq1,\quad\forall s\geq 0,
\end{equation}
and for $x\in \mathbb R^2$ define the $C^1$-deformation $T $  by
$$	T (x)=\tau(|x|^2) x.$$
Then by direct computations, we have
\begin{equation}\label{DT}
	DT (x)=\tau(|x|^2)\left(I_2+\frac{1}{h^2+h\sqrt{h^2+|x|^2}}\begin{pmatrix}
		x_1^2&x_1x_2\\x_1x_2&x_2^2
	\end{pmatrix}\right), \ \  \text{where}\ \  I_2=\begin{pmatrix}
		1 & 0 \\ 0&1
	\end{pmatrix}.
\end{equation}

Our first discovery in this paper is that, formally speaking,  the $N$-helical vortex system with the given initial conditions as above   will approximately move according to the pattern of an ODE dynamical system. To our knowledge, this dynamical system has not been reported in other literature.
\begin{proposition}[Formal result]\label{pro-ds}
	  Assume that the $i$-th vortex helix $\omega_{i,\varepsilon}(x,t)$ remains concentrated around $x_0+\frac{ P_{i,\varepsilon}(t)}{|\ln\ep| } $ for $t\in \left(0, \frac{T}{|\ln\varepsilon|^{1-b}}\right)$. Then  the rescaled center  $P_{i,\varepsilon}(t)$ converges as $\ep \to 0$ to a limiting trajectory $t\mapsto P_i(t)$, and after applying a transformation $\tilde{P}_i(t)=DT(x_0)P_i(t)$, this trajectory satisfies
 \begin{equation}\label{ds}
 	\begin{cases}
 		\partial_t \tilde{P}_i=A\sum_{j\neq i}a_j\frac{(\tilde{P}_i-\tilde{P}_j)^\perp}{|\tilde{P}_i-\tilde{P}_j|^2}-a_iB\begin{pmatrix}
 			0\\1
 		\end{pmatrix},\\
 		\tilde{P}_i(0)=\tilde{P}_i^0, \quad i=1,...,N,
 	\end{cases}
 \end{equation}
 where  $\tilde{P}_i^0=DT(x_0)P_i^0$($i=1,\ldots,N$), $A$ and $B$ are constants given by
 \begin{equation}\label{def ab}
 	A=\frac{\sqrt{h^2+r_0^2}\left(r_0^2+h^2+h\sqrt{h^2+r_0^2}\right)}{2\pi h\left(h^2+h\sqrt{h^2+r_0^2}\right)},\quad B=\frac{\tau(r_0^2)r_0}{4\pi h\sqrt{h^2+r_0^2}}.
 \end{equation}
\end{proposition}

In order to maintain clarity, the statements of Proposition \ref{pro-ds} are presented in a formal yet non-rigorous manner. For further details, please refer to subsection \ref{sec2-2} below.

We emphasize that  Proposition \ref{pro-ds} is a formal result since we have assumed that the vortex helices ``remain concentrated" without a proof.  The primary objective of this paper is to rigorously establish sufficient concentration of these helical vortices for positive times, thereby demonstrating that the dominant component of their motion corresponds to the dynamical system (\ref{ds}).

Now we can state the first main results of the paper.
\begin{theorem}\label{thm-main}
	Assume that the initial data $\omega_\varepsilon(x,0)$ verifies conditions (\ref{1}), (\ref{2}), (\ref{3})  and (\ref{4}).  Denote by $(\tilde{P}_1(t),...,\tilde{P}_N(t))$, $t\in[0,T^*)$ the maximal solution to the   Cauchy problem  \eqref{ds} with  initial data $\tilde{P}_i^0=DT(x_0)P_i^0$.  Let ${P}_i(t)=DT(x_0)^{-1}\tilde{P}_i(t)$. Then for any fixed $\rho>0$ such that the closed disks $\overline{B(x_0+ P_i^0 |\ln \varepsilon|^{-1}, 2\rho |\ln \varepsilon|^{-1})}$ are mutually disjoint,  there exists $T_{\rho}\in (0,T^*)$ such that for any $\varepsilon$ small enough and $t\in[0,T_{\rho}]$ the following holds true.
	\begin{enumerate}
		\item The support of each component of the corresponding solution to \eqref{HEuler} with initial data $\omega_\varepsilon(x,0)$, as denoted in \eqref{5}, satisfies $$\Lambda_{i,\varepsilon}(t):=\mathrm{supp}\ \omega_{i,\varepsilon}(\cdot,t)\subseteq B(x_0+{P_i(t)}/{|\ln \varepsilon|},{\rho}/{|\ln \varepsilon|}),$$ and the disks $B(x_0+{P_i(t)}/{|\ln \varepsilon|},{2\rho}/{|\ln \varepsilon|})$ are mutually disjoint for $\ep$ sufficiently small.
		\item There exist $(P^{\varepsilon}_1(t),...,P^{\varepsilon}_N(t))$ and $\rho_{\varepsilon}>0$ such that
		\begin{equation*}
			\lim_{\varepsilon \xrightarrow{}0}|\ln \varepsilon|^{1+b}\int_{B\left(P^{\varepsilon}_i(t),\frac{\rho_{\varepsilon}}{|\ln \varepsilon|}\right)}\omega_{i,\varepsilon}(x,t)dx=a_i \quad\forall i=1,...,N,
		\end{equation*}
		with\begin{equation*}
			\lim_{\varepsilon \xrightarrow{}0}\rho_{\varepsilon}=0\quad and\quad\ \lim_{\varepsilon \xrightarrow{}0}|\ln \varepsilon|\left|P^{\varepsilon}_i(t)-\left(x_0+\frac{P_i(t)}{|\ln \varepsilon|}\right)\right|=0,\quad \forall i=1,...,N.
		\end{equation*}
	\end{enumerate}
\end{theorem}
\begin{remark}
	As mentioned above, dynamics of multiple vortex helices was also considered in \cite{DLM}. However, we remark that their assumptions on initial data are quite different with ours. In particular, the  helical vortices in \cite{DLM}  are assumed to be initially separated by   distances with a positive lower bound independent  of  $\ep$, so that the  mutual interaction  can be neglected compared to their self-interaction and hence each vortex helix in the system dynamics following the same law as  those of isolated vortex helix. See Theorem 1 and Section 3 in \cite{DLM}. It can be seen  from \eqref{4} that in our settings for initial vorticity, the mutual distances tend to $0$  as $\ep\to0$, and the self-induced motion of the
	vortex helices is of the same order as the velocity induced by the interaction of vortex helices. This enables us to derive the new discovered dynamical system \eqref{ds} describing the evolution of multiple helical vortex filaments.
\end{remark}
\begin{remark}
    Since it was proved in \cite{DLM} that Green's function for $\mathcal L_K$  can be decomposed in a similar form as \eqref{decompG} in bounded domains  with $C^{1,1}$ boundary, our results in Theorem \ref{thm-main} can also be established in bounded domains without modifying the proof.
\end{remark}
\begin{remark}
	Our Theorem \ref{thm-main} concerns the evolution of helical vortices with general initial data. The time $T_\rho$ could be extremely small. We conjecture that by selecting some specific initial data, $T_\rho$ can be extended to arbitrary $T<T^*$ as the case for vortex rings in \cite{DDMW}.
\end{remark}

As in \cite{BCM3} for vortex rings, the proof of Theorem \ref{thm-main} relies on the conservation of energy, an iteration argument and a formal principle  ``The divergence of the fluid kinetic energy at a specific rate as
$\ep\to0$ suggests that the vorticity becomes concentrated within a vanishingly small region".  Nevertheless, in the cases of helical vortices, many new difficulties arise as the operator $\mathcal L_K:=\mathrm{div}(K(x)\nabla)$ is anisotropic, leading to the presence of a nonlinear transformation $T$ in the main order term of the corresponding Green's function (see \eqref{decompG} below). Consequently, several crucial cancellations that fundamentally rely on symmetry properties in the calculations of \cite{BCM3} (as exemplified by the proof of Proposition 4.2 in their framework) cease to hold in our generalized context.
The aforementioned challenges are resolved in our study by the introduction of novel transformed components: specifically, a coordinate-transformed cut-off function $W_{R,\zeta}\big(DT(x_0)(x-B^{i,\varepsilon}(t))\big)$ and   distance function $ |DT(x_0)(x(t)-B^{i,\varepsilon}(t)) |$. This fundamentally differs from conventional methods documented in existing works (e.g., \cite{BCM3, DLM}) that directly employ untransformed coordinates $W_{R,\zeta}(x-B^{i,\varepsilon}(t))$ and $ |x(t)-B^{i,\varepsilon}(t) |$.  We believe that the using of the transformation $DT(x_0)$ is necessary and essential since $\mathcal L_K$ is anisotropic.

\subsection{An example for leapfrogging helical vortices} The dynamical system \eqref{ds} has periodic solutions, see for example our later analysis in Section \ref{leapfrogging}.  In view of the numerical observation for leapfrogging helical vortices in e.g.  \cite{CL, DSR, HHSB}, a natural question arises regarding the applicability of Theorem \ref{thm-main} in establishing rigorous confirmation of this observed phenomenon. This requests us to enlarge the time $T_\rho$ to cover several periods.

However, the second-order elliptic operator $\mathcal L_K:=\mathrm{div}(K(x)\nabla)$ exhibits progressive degeneracy in neighborhoods of a given point $x_0=(r_0,0)$ as $r_0\to+\infty$. Therefore, it seems hard to apply directly the approach developed in Section 7 of  \cite{BCM3} by taking the main radii $r_0$ sufficiently large to extend the time $T_\rho$ in Theorem \ref{thm-main} to guarantee leapfrogging of two helical vortices. We observe that the time $T_\rho$ can still be extended by decreasing the initial distance of the two helical vortices. This strategy ultimately enables us to provide the first rigorous mathematical derivation of leapfrogging motion for helical vortices, which forms the second main results of this paper.
\begin{theorem}\label{thm-leapfrogging}
	Suppose that $N=2$ and  $a_1+a_2\not=0$. For any given integer $k>0$, there exists $\rho_0'$ sufficiently small such that if the initial distance $|P_1^0-P_2^0|\leq \rho_0'$, and $4\rho<\frac{|P_1^0-P_2^0|(h^2+h\sqrt{h^2+r_0^2})}{r_0^2+h^2+h\sqrt{h^2+r_0^2}}$, then the time $T_\rho$ in Theorem \ref{thm-main} can be extended to $T'_\rho\geq T_\rho$ so that the conclusions of Theorem  \ref{thm-main} still hold for any $t\in[0, T_\rho']$. Moreover, it holds that
	$$ T_\rho'>k T_E, $$
where $T_E$ stands for the period of motion of the solution for \eqref{ds} with initial data $\tilde{P}_i^0=DT(x_0)P_i^0$ (i=1,2), which is explicated later in Section \ref{leapfrogging}.
\end{theorem}

\begin{remark}
	 Similar to the case of vortex rings considered in \cite{DDMW, DHLM}, leapfrogging phenomena can also occur when each vortex helix has identical constant vorticity mass and mutual separation on the order $1/\sqrt{|\ln\ep|}$. We will investigate this case using the approaches developed in this paper and \cite{DHLM} in our forthcoming work \cite{CFQW}.  
\end{remark}

This paper is organized as follows: In Section 2, we present a useful decomposition for Green's function corresponding to $\mathcal L_K$ and a formal derivation of the dynamical system \eqref{ds}. Section 3 is devoted to the estimates for the localization of vorticity by the conversation of energy and an iteration argument. We determine the location of supports of vortices and derive rigorously the dynamical system \eqref{ds} in Section 4, which also proves the main results Theorem \ref{thm-main}. In the final section, we discuss the special cases of two vortex helices and show that  leapfrogging occurs  provided that the initial distances are sufficiently small via an iterative argument.

\section{Preliminaries and formal derivation of the dynamical system}
\label{sec:preliminaries}

In this section we collect some useful tools such as the decomposition of  Green's function and give a formal  derivation of the dynamical system \eqref{ds} for $N$-interacting helical vortices as $\ep\to 0$.

\subsection{The decomposition of Green's function}
The existence of Green's function $\mathcal G_K$ for the operator $\mathcal{L}_K=\mathrm{div}(K(x)\nabla)$  with $K$ defined by \eqref{def K} in $\mathbb R^2$ has been established in \cite{CFLQ, CLQW} with an explicit formula given by Fourier series in \cite{CFLQ}. A key ingredient throughout our proof is  the following refined decomposition of Green's function $\mathcal G_K$ obtained in \cite{DLM} for the cases of bounded domains.

In what follows, we shall denote $H(x,y)=\frac{(\det K(x) \det K(y))^{-\frac{1}{4}}}{2\pi}$ and $|X|:=\sqrt{|x|^2+h^2}$.
\begin{proposition}\label{decompG}
	The  Green's  function $\mathcal G_K$ can be decomposed as
	$$\mathcal G_K(x,y)=G_K(x,y)+S_K(x,y),$$
where for $x\not= y\in \mathbb R^2$ the singular part
	\begin{equation}\label{GK}
		G_K(x,y)=H(x,y)\ln  |T (x)-T (y)|,
	\end{equation}
and the relative smooth part $S_K\in W^{1,\infty}_{loc}(\mathbb R^2\times \mathbb R^2)$.
\end{proposition}

\subsection{Formal derivation of the dynamical system}\label{sec2-2}

In this section, we will   derive the dynamical system \eqref{ds} in a formal way by assuming that vorticity remains concentrated as time evolves. More precisely, we will assume that the initial condition $\omega_\varepsilon(x,0)$ verifies   (\ref{1}), (\ref{2}), (\ref{3}) and (\ref{4}). By the transport nature of \eqref{HEuler}, one has
\begin{equation*}
	|\omega_{i,\varepsilon}(x,t)|\leq \frac{M}{\ep^2}, \ \  \  \int   \omega_{i,\varepsilon}(x,t)dx=\gamma_i=\frac{a_i}{|\ln\ep|^{1+b}}.
\end{equation*}

Denote the center of mass for each component
\begin{equation}\label{cm}
	x_{i,\varepsilon}(t)=\frac{1}{\gamma_i}\int_{\mathrm{supp}\	\omega_{i,\varepsilon}(\cdot,t)} x\omega_{i,\varepsilon}(x,t)dx.
\end{equation}

Note the total mass of these vortex helices is of order $1/|\ln\ep|^{1+b}$. According to the analysis in \cite{DLM}, if these $N$ helical vortex tubes are regarded as a whole, they will have a rotational motion around the origin of order $1/|\ln\ep|^{b}$.   Therefore, within a finite time of order $1/|\ln\varepsilon|^{1-b}$, the vortex helices can only move a distance of  $|x_{i, \ep}-x_0|=O(1/|\ln\ep|)$ from the starting point. Thus, it is reasonable to assume that the center can be written in the form  $x_{i,\varepsilon}(t)=x_0+P_{i,\varepsilon}(t)/|\ln \varepsilon|$ for some point $P_{i, \ep}$.

Now suppose that for $0<t<T$, there exist points $P_{i,\varepsilon}(t)$ and some large constant $C>0$ with
\begin{equation}\label{assume supp} \begin{split}
 x_{i,\varepsilon}(t)=x_0+\frac{P_{i,\varepsilon}(t)}{|\ln \varepsilon|},\ \   |x_{i,\varepsilon}(t)|\leq C,\   \ \mathrm{supp}\	 \omega_{i,\varepsilon}(\cdot,t)\subset B(x_{i,\varepsilon}(t), C \varepsilon), \   i=1,\ldots, N, \\
P_{i,\varepsilon}(0)=P^0_i\qquad \text{and}\ \  \quad|P_{i,\varepsilon}(t)-P_{j,\varepsilon}(t)|\geq C^{-1}>0, \ \  \forall\ i\not=j.
\end{split}\end{equation}

In what follows, we simply denote $\nabla=\nabla_x$ as the derivative with respect to variable $x$. 
\begin{proof}[Proof of Proposition \ref{pro-ds}]
By the definition of $x_{i,\varepsilon}(t)$, we have
\begin{equation}\label{dcm}
	\dot{x}_{i,\varepsilon}(t)=\frac{1}{|\ln\varepsilon|^{1-b}\gamma_i}\iint\omega_{i,\varepsilon}(x,t)\sum^N_{j=1}\omega_{j,\varepsilon}(y,t)\nabla^\perp \mathcal{G}_K (x,y)dxdy.
\end{equation}

It follows from  Proposition \ref{decompG} and straightforward calculations that
\begin{equation}\label{decom DG}
	\nabla^\perp \mathcal{G}_K(x,y) =\nabla^\perp H(x,y)\ln|T(x)-T(y)|+H(x,y)\left(DT(x)\frac{T(x)-T(y)}{|T(x)-T(y)|^2}\right)^\perp+\nabla^\perp S_{K}(x,y).
\end{equation}
Since $S_K\in W^{1,\infty}_{loc}(\mathbb R^2\times \mathbb R^2)$, we infer from assumption \eqref{assume supp}  that
\begin{equation}\label{2-8}
	 \left|\frac{1}{\gamma_i}\iint\omega_{i,\varepsilon}(x,t)\sum^N_{j=1}\omega_{j,\varepsilon}(y,t)\nabla^\perp S_K (x,y)dxdy \right|=O\left(\frac{1}{|\ln \varepsilon|^{1+b}}\right).
\end{equation}
Using the Taylor's expansion, for $x\in \mathrm{supp}\ \omega_{i,\varepsilon}(\cdot,t)$ and $y\in \mathrm{supp}\ \omega_{j,\varepsilon}(\cdot,t)$, we find
\begin{equation}\label{Taylor T}
	T(x)-T(y)=\frac{DT(x_0)(P_{i,\varepsilon}-P_{j,\varepsilon})}{|\ln\ep|}+O\left(\frac{1}{|\ln \varepsilon|^2}\right),\ \  H(x,y)=H(x_0, x_0)+O\left(\frac{1}{|\ln \varepsilon|}\right).
\end{equation}
Then, we compute and get
\begin{align}\label{2-9}
	&\quad  \frac{1}{\gamma_i}\sum_{j\neq i}\iint\omega_{i,\varepsilon}(x,t)\omega_{j,\varepsilon}(y,t)(\nabla^\perp \mathcal{G}_K (x,y)-\nabla^\perp S_K (x,y))dxdy  \\
	&=\sum_{j\neq i}\gamma_j\nabla^\perp H(x_0,x_0)\ln \left|DT(x_0)\frac{P_{i,\varepsilon}-P_{j,\varepsilon}}{|\ln \varepsilon|}\right| \nonumber \\
	&\quad +\sum_{j\neq i}\gamma_j H(x_0,x_0)|\ln \varepsilon|\frac{(DT(x_0)^2(P_{i,\varepsilon}-P_{j,\varepsilon}))^\perp}{|DT(x_0)(P_{i,\varepsilon}-P_{j,\varepsilon})|^2} \nonumber\\
	&\quad +O\left(\frac{1}{|\ln \varepsilon|^{1+b}}\right)+O\left(\frac{\ln |\ln \varepsilon|}{|\ln \varepsilon|^{1+b}}\right)  \nonumber\\
	&=\sum_{j\neq i}\gamma_j H(x_0,x_0)|\ln \varepsilon|\frac{(DT(x_0)^2(P_{i,\varepsilon}-P_{j,\varepsilon}))^\perp}{|DT(x_0)(P_{i,\varepsilon}-P_{j,\varepsilon})|^2} +O\left(\frac{\ln |\ln \varepsilon|}{|\ln \varepsilon|^{1+b}}\right).\nonumber
\end{align}
It remains to calculate $	\frac{1}{\gamma_i}\iint\omega_i(x,t)\omega_i(y,t) (\nabla^\perp \mathcal{G}_K (x,y)-\nabla^\perp S_K (x,y)) dxdy$.  Indeed, by symmetry properties, we obtain
\begin{align}\label{2-10}
	&\quad  \left| 	\frac{1}{\gamma_i}\iint\omega_{i,\varepsilon}(x,t)\omega_{i,\varepsilon}(y,t) H(x,y)\left(DT(x)\frac{T(x)-T(y)}{|T(x)-T(y)|^2}\right)^\perp dxdy  \right|\\
	&=\left| 	\frac{1}{2\gamma_i}\iint\omega_{i,\varepsilon}(x,t)\omega_{i,\varepsilon}(y,t) H(x,y)\left(( DT(x)-DT(y)) \frac{T(x)-T(y)}{|T(x)-T(y)|^2}\right)^\perp dxdy\right|\nonumber \\
	&\leq 	\frac{C}{ \gamma_i}\iint\omega_{i,\varepsilon}(x,t)\omega_{i,\varepsilon}(y,t) H(x,y) dx dy= O\left(\frac{1}{|\ln \varepsilon|^{1+b}}\right).\nonumber
\end{align}
By the smoothness of $T(x)$, it can be verified that $$\ln \frac{|T(x)-T(y)|}{|x-y|}\leq C, \ \  \forall\ x, y\in  \mathrm{supp}\ \omega_{i,\varepsilon}(\cdot,t).$$ Therefore, we get
\begin{align}\label{2-11}
		&\quad  \frac{1}{\gamma_i}\iint\omega_{i,\varepsilon}(x,t)\omega_{i,\varepsilon}(y,t)\nabla^\perp H(x,y)\ln|T(x)-T(y)|dxdy\\
		&=\frac{1 }{\gamma_i}\iint\omega_{i,\varepsilon}(x,t)\omega_{i,\varepsilon}(y,t)\nabla^\perp H(x,y)\ln|x-y|dxdy +O\left(\frac{1}{|\ln \varepsilon|^{1+b}}\right)\nonumber\\
		&=\frac{\nabla^\perp H(x_0,x_0)+O(1/|\ln\ep|) }{\gamma_i}\iint\omega_{i,\varepsilon}(x,t)\omega_{i,\varepsilon}(y,t)\ln|x-y|dxdy +O\left(\frac{1}{|\ln \varepsilon|^{1+b}}\right).\nonumber
\end{align}
We need to compute $\frac{1 }{\gamma_i}\iint\omega_{i,\varepsilon}(x,t)\omega_{i,\varepsilon}(y,t)\ln|x-y|dxdy$. On the one hand, for any $x, y \in \mathrm{supp}\ \omega_{i,\varepsilon}(\cdot,t)$, we infer from assumption \eqref{assume supp} that $|x-y|\leq 2C\ep$. Thus, we obtain
\begin{equation}\label{upbd}
	\frac{1 }{\gamma_i}\iint\omega_{i,\varepsilon}(x,t)\omega_{i,\varepsilon}(y,t)\ln|x-y|dxdy\leq -\gamma_i |\ln\ep|+O(|\gamma_i|).
\end{equation}
On the other hand, setting $r_\ep:=\ep \sqrt{\frac{\gamma_i}{\pi M}}$, we then get by using the rearrangement inequality
 $$\frac{1 }{\gamma_i}\iint\omega_{i,\varepsilon}(x,t)\omega_{i,\varepsilon}(y,t)\ln|x-y|dxdy\geq  \frac{2\pi M }{\ep^2}  \int_{0}^{r_\ep} s \ln s\  ds=\gamma_i \ln\ep+O\left(\frac{\ln |\ln \varepsilon|}{|\ln \varepsilon|^{1+b}}\right).$$
Hence we conclude that
\begin{equation}\label{2-12}
	\frac{1 }{\gamma_i}\iint\omega_{i,\varepsilon}(x,t)\omega_{i,\varepsilon}(y,t)\ln|x-y|dxdy= - \gamma_i |\ln\ep|+O\left(\frac{\ln |\ln \varepsilon|}{|\ln \varepsilon|^{1+b}}\right).
\end{equation}

A combination of identities \eqref{dcm}--\eqref{2-12}  yields
\begin{equation*}
	\begin{split}
		 \frac{\dot P_{i,\varepsilon}(t)}{|\ln \varepsilon|}&=\frac{1}{|\ln \varepsilon|}\sum_{j\neq i}a_j H(x_0,x_0)\frac{(DT(x_0)^2(P_{i,\varepsilon}-P_{j,\varepsilon}))^\perp}{|DT(x_0)(P_{i,\varepsilon}-P_{j,\varepsilon})|^2}\\
		&-\frac{1}{|\ln \varepsilon|}a_i\nabla^\perp H(x_0,x_0)+O\left(\frac{\ln |\ln \varepsilon|}{|\ln \varepsilon|^2}\right).
	\end{split}
\end{equation*}
Multiplying the above equation by $|\ln\ep|$ and letting $\varepsilon\xrightarrow{}0$, at least in a formal way, we get the desired dynamical system as follows:
\begin{equation}
	{\dot P_i(t)}=\sum_{j\neq i}a_j H(x_0,x_0)\frac{(DT(x_0)^2(P_i-P_j))^\perp}{|DT(x_0)(P_i-P_j)|^2}
	-a_i\nabla^\perp H(x_0,x_0).
\end{equation}
Recalling that $x_0=(r_0,0)$, by definitions, we have
\begin{equation*}
	\begin{split}
		&  H(x_0,x_0)=\frac{\sqrt{h^2+r_0^2}}{2\pi h}, \ \  \  \nabla^\perp H(x_0,x_0)= \frac{(0,r_0)}{4\pi h\sqrt{h^2+r_0^2}},\\
		&DT(x_0)=\tau(r_0^2)\left(I_2+\frac{1}{h^2+h\sqrt{h^2+r_0^2}}\begin{pmatrix}
			r_0^2&0\\
			0&0
		\end{pmatrix}\right).
	\end{split}
\end{equation*}
If we denote $\tilde{P}_i(t)=DT(x_0)P_i(t)$, then we obtain for $i=1,...,N,$
\begin{equation*}
	\left\{
	\begin{array}{ll}
		\partial_t \tilde{P}_i=A\sum_{j\neq i}a_j\frac{(\tilde{P}_i-\tilde{P}_j)^\perp}{|\tilde{P}_i-\tilde{P}_j|^2}-a_iB\begin{pmatrix}
			0\\1
		\end{pmatrix},  \\
		\tilde{P}_i(0)=DT(x_0)P^0_i,
	\end{array}
	\right.
\end{equation*}
where $A$   and $B$ are constants defined by \eqref{def ab}. Therefore,  we obtain formally  the dynamical system \eqref{ds} and conclude the proof.
\end{proof}
\begin{remark}
    Notice  that the strong localization assumption $\mathrm{supp}\	\omega_{i,\varepsilon}(\cdot,t)\subset B(x_{i,\varepsilon}(t), C \varepsilon)$ is merely used in the derivation of upper bound \eqref{upbd}. In what follows, we will derive an estimate for $\frac{1 }{\gamma_i}\iint\omega_{i,\varepsilon}(x,t)\omega_{i,\varepsilon}(y,t)\ln|x-y|dxdy$  by using the conservation of energy, so that   only a relatively weaker localization property   needs to be established.
\end{remark}

\section{Concentration estimates of the vorticity}
In the previous section, we have derived the dynamic system \eqref{ds} in a  formal way by assuming \eqref{assume supp}. In the rest of this paper, we aim to rigorously prove \eqref{ds}. We first  establish some estimates for the concentration properties for vorticity by using  conservation of energy.

Let $\rho$ be the constant in Theorem \ref{thm-main}. Since $|P^0_i-P^0_j|>4\rho$ for any $i\neq j$, by the continuity of solution
we can find $T\in (0,T^*)$ satisfying
\begin{equation}\label{defT}
    \min\limits_{i\neq j}\min\limits_{t\in[0,T]}|P_i(t)-P_j(t)|\geq 4\rho.
\end{equation}
Let
\begin{equation}\label{defd}
    \overline{d}=\max\limits_{t\in[0,T]}\left\{\max\limits_{1\leq i\leq N}|P_i(t)|,\max_{i\neq j}|P_i(t)-P_j(t)|\right\}.
\end{equation}
Define
\begin{equation}\label{defTep}
    T_\varepsilon=\max\left\{t\in[0,T]:\Lambda_{i,\varepsilon}(s)\subset B\left(x_0+\frac{1}{|\ln \varepsilon|}P_i(s),\frac{1}{|\ln \varepsilon|}\rho\right), \forall s\in [0,t],\forall i\in\{1,\ldots,N\}\right\}.
\end{equation}
Without loss of generality, hereafter we assume that $\varepsilon_0<\rho$ so that $T_\varepsilon>0$ for any $\varepsilon\in(0,\varepsilon_0)$ by continuity. It is obvious that
\begin{equation}\label{3-4}
    \begin{cases}
        |x|\leq \frac{\overline{d}+\rho}{|\ln \varepsilon|}+|x_0|,&\quad \forall x\in\Lambda_{i,\varepsilon}(t)\quad \forall t\in[0,T_\varepsilon]\quad\forall i,\\
        |x-y|\geq \frac{2\rho}{|\ln \varepsilon|},&\quad \forall x\in\Lambda_{i,\varepsilon}(t)\quad \forall y\in\Lambda_{j,\varepsilon}(t)\quad\forall t\in[0,T_\varepsilon]\quad\forall i\neq j.
    \end{cases}
\end{equation}
\subsection{Energy estimate.}
The kinetic energy is defined as follows:
\begin{equation}\label{defEn}
    E(t)=-\iint\mathcal{G}_K (x,y)\omega_\varepsilon(x,t)\omega_\varepsilon(y,t)dxdy,
\end{equation}
which can be decomposed in the following way
\begin{equation}\label{decompE}
    E(t)=\sum_iE_i(t)+2\sum_{i>j}E_{i,j}(t),
\end{equation}
where  the self-interaction energy is given by
\begin{equation}
    E_i(t)=-\iint\mathcal{G}_K (x,y)\omega_{i,\varepsilon}(x,t)\omega_{i,\varepsilon}(y,t)dxdy,
\end{equation}
and the interaction energy is given by
\begin{equation}
    E_{i,j}(t)=-\iint\mathcal{G}_K (x,y)\omega_{i,\varepsilon}(x,t)\omega_{j,\varepsilon}(y,t)dxdy.
\end{equation}
To simplify notations, we set
\begin{equation}
    \overline{\gamma}=\sum_i|\gamma_i|, \quad a=\sum_i|a_i|.
\end{equation}

\begin{lemma}\label{6}
    There exists $C_1=C_1(|x_0|,\overline{d},h,a)>0$ such that, for any $\varepsilon\in(0,\varepsilon_0)$ and $t\in[0, T_\ep]$,
    \begin{equation}\label{estimateE}
        \sum_iE_i(t)\geq\sum_i\frac{|a_i|^2}{|\ln \varepsilon|^{1+2b}}\frac{\sqrt{r_0^2+h^2}}{2\pi h}-C_1\frac{\ln |\ln \varepsilon|}{|\ln \varepsilon|^{2+2b}}.
    \end{equation}
\end{lemma}
\begin{proof}
    Denote $|X|=\sqrt{|x|^2+h^2}$ and  $|X_0|=\sqrt{|x_0|^2+h^2}$. By using  arguments similar as the previous section (in particular, the proof of \eqref{2-11} and \eqref{2-12}), one can show that
    \begin{equation*}
    \begin{split}
        \sum_iE_i(0)&=\sum_i|\gamma_i|^2\frac{|X_0|}{2\pi h}|\ln \varepsilon|+O(\gamma_i^2)\\
        &\geq\sum_i|a_i|^2\frac{|X_0|}{2\pi h}\frac{1}{|\ln \varepsilon|^{1+2b}}-C'_1\frac{1}{|\ln \varepsilon|^{2+2b}},
    \end{split}
    \end{equation*}
    where $C_1'$ is a constant depending on $|x_0|,\overline{d},h,a$. \\
    Now, for the interaction energy, by our choice of $T_\ep$, we get that for $\varepsilon\in(0,\varepsilon_0)$, $t\in[0, T_\ep]$ and $i\not=j$,
    \begin{equation*}
    \begin{split}
        |E_{i,j}(t)|&\leq \left|\iint\frac{\sqrt{|X||Y|}}{2\pi h}\ln |T(x)-T(y)|\omega_{i,\varepsilon}(x,t)\omega_{j,\varepsilon}(y,t)dxdy\right|+O(|\gamma_i\gamma_j|)\\
        &=\left| \iint\frac{\sqrt{|X||Y|}}{2\pi h}\ln |x-y|\omega_{i,\varepsilon}(x,t)\omega_{j,\varepsilon}(y,t)dxdy\right|+O(|\gamma_i\gamma_j|)\\
        &\leq C_1''|\gamma_i\gamma_j|\ln \left|\frac{\ln \varepsilon}{2\rho}\right|+O\left(\frac{1}{|\ln \varepsilon|^{2+2b}}\right)=O\left(\frac{\ln |\ln \varepsilon|}{|\ln \varepsilon|^{2+2b}}\right).
     \end{split}
    \end{equation*}

    So by the conservation of energy (i.e., $E(t)=E(0)$), we obtain
    \begin{equation*}
        \begin{split}
            \sum_iE_i(t)&=\sum_iE_i(0)+2\sum_{i>j}(E_{i,j}(0)-E_{i,j}(t))\\
            &\geq\sum_i\frac{|a_i|^2}{|\ln \varepsilon|^{1+2b}}\frac{|X_0|}{2\pi h}-C_1\frac{\ln |\ln \varepsilon|}{|\ln \varepsilon|^{2+2b}},
        \end{split}
    \end{equation*}
    which is the desired estimate \eqref{estimateE} and therefore finishes the proof.
\end{proof}
We take  $\varepsilon_0$ smaller such that $\ln |\ln \varepsilon|>1$ for any $\varepsilon\in(0,\varepsilon_0)$.
\begin{proposition}
    There exists $C_2=C_2(|x_0|,\overline{d},h,a)>0$ such that, for any $\varepsilon\in(0,\varepsilon_0)$,
    \begin{equation}\label{7}
    \begin{split}
        \mathcal{G}_i(t):=&\iint\omega_{i,\varepsilon}(x,t)\omega_{i,\varepsilon}(y,t)\ln \left(\frac{|\ln \varepsilon|^{\frac{1+b}{2}}|x-y|}{\varepsilon}\right)\Pi \left(|x-y|\geq\frac{\varepsilon}{|\ln \varepsilon|^{\frac{1+b}{2}}}\right)dxdy\\
        &\leq C_2\frac{\ln |\ln \varepsilon|}{|\ln \varepsilon|^{2+2b}}, \quad \forall t\in [0,T_\varepsilon],\quad \forall i=1,...,N,
    \end{split}
    \end{equation}
     where $\Pi (\cdot)$ stands for the indicator function of a subset.
\end{proposition}
\begin{proof}
    In view that $S_{K} (x,y)\in W_{loc}^{1,\infty}$  and $\ln \frac{|T(x)-T(y)|}{|x-y|}\leq C$, there exist positive constants $C_2'$ and  $C_2''$ depending on $|x_0|,a,\overline{d},h$ such that for all $ t\in[0,T_\varepsilon]$ and for all $\varepsilon\in(0,\varepsilon_0)$,
    \begin{equation*}
        \begin{split}
            E_i(t)&\leq-\iint\frac{\sqrt{|X||Y|}}{2\pi h}\ln |T(x)-T(y)|\omega_{i,\varepsilon}(x,t)\omega_{i,\varepsilon}(y,t)dxdy+C_2'\frac{1}{|\ln \varepsilon|^{2+2b}}\\
            &\leq\iint\frac{\sqrt{|X||Y|}}{2\pi h}\ln |x-y|^{-1}\omega_{i,\varepsilon}(x,t)\omega_{i,\varepsilon}(y,t)dxdy+C_2''\frac{1}{|\ln \varepsilon|^{2+2b}}\\
    &:= G_i^{(1)}(t)-G^{(2)}_i(t)-G_i^{(3)}(t)+C_2''\frac{1}{|\ln \varepsilon|^{2+2b}},
     \end{split}
    \end{equation*}
 where
    \begin{equation*}
        \begin{split}
           & G_i^{(1)}(t)=\iint\frac{\sqrt{|X||Y|}}{2\pi h}\ln \left(\frac{|\ln \varepsilon|^{\frac{1+b}{2}}}{\varepsilon}\right)\omega_{i,\varepsilon}(x,t)\omega_{i,\varepsilon}(y,t)dxdy,\\
           &G^{(2)}_i(t)=\iint\frac{\sqrt{|X||Y|}}{2\pi h}\ln \left(\frac{|\ln \varepsilon|^{\frac{1+b}{2}}|x-y|}{\varepsilon}\right)\omega_{i,\varepsilon}(x,t)\omega_{i,\varepsilon}(y,t)\Pi \left(|x-y|\geq\frac{\varepsilon}{|\ln \varepsilon|^{\frac{1+b}{2}}}\right)dxdy,\\
           &G^{(3)}_i(t)=\iint\frac{\sqrt{|X||Y|}}{2\pi h}\ln \left(\frac{|\ln \varepsilon|^{\frac{1+b}{2}}|x-y|}{\varepsilon}\right)\omega_{i,\varepsilon}(x,t)\omega_{i,\varepsilon}(y,t)\Pi \left(|x-y|<\frac{\varepsilon}{|\ln \varepsilon|^{\frac{1+b}{2}}}\right)dxdy.
        \end{split}
    \end{equation*}
    By the definition of $T_\varepsilon$ and direct computation we deduce
    \begin{equation}
        \begin{split}
            G^{(1)}_i(t)&\leq\frac{|X_0|}{2\pi h}\gamma_i^2\ln \left(\frac{|\ln \varepsilon|^{\frac{1+b}{2}}}{\varepsilon}\right)+C_{2,1}\frac{1}{|\ln \varepsilon|^{3+2b}}\ln \left(\frac{|\ln \varepsilon|^{\frac{1+b}{2}}}{\varepsilon}\right)\\
            &\leq\frac{|X_0|}{2\pi h}a_i^2\frac{1}{|\ln \varepsilon|^{1+2b}}+C_{2,1}'\frac{\ln |\ln \varepsilon|}{|\ln \varepsilon|^{2+2b}}.
        \end{split}
    \end{equation}
    Denote $c\wedge d=\min\{c,d\}$ for any $c,d\in \mathbb R$. Let ${r}_{\ep, i}=\sqrt{\frac{|\gamma_i|}{M\pi}} \ep$ so that $
        \frac{M}{\varepsilon^2}\pi{r}_{\ep, i}^2=|\gamma_i|,$
    then we deduce from the rearrangement inequality that
    \begin{equation*}
    \begin{split}
        &\int\omega_{i,\varepsilon}(y,t) \ln \left(\frac{\varepsilon}{|\ln \varepsilon|^{\frac{1+b}{2}}|x-y|}\right)\Pi \left(|x-y|<\frac{\varepsilon}{|\ln \varepsilon|^{\frac{1+b}{2}}}\right)dy\\
        &\leq\frac{M}{\varepsilon^2}\int_0^{{r}_{\ep, i}\wedge \frac{\varepsilon}{|\ln \varepsilon|^{(1+b)/2}}}2\pi r\ln \left(\frac{\varepsilon}{|\ln \varepsilon|^{\frac{1+b}{2}}r}\right)dr\\
        &=\frac{2\pi M}{\varepsilon^2}\left[\frac{\left({r}_{\ep, i}\wedge \frac{\varepsilon}{|\ln \varepsilon|^{(1+b)/2}}\right)^2}{4}-\frac{\left({r}_{\ep, i}\wedge \frac{\varepsilon}{|\ln \varepsilon|^{(1+b)/2}}\right)^2}{2}\ln \frac{|\ln\varepsilon|^{\frac{1+b}{2}}\left({r}_{\ep, i}\wedge \frac{\varepsilon}{|\ln \varepsilon|^{(b+1)/2}}\right)}{\varepsilon}\right]\\
        &\leq C_{2,2}\frac{1}{|\ln \varepsilon|^{1+b}},
    \end{split}
    \end{equation*}
    which gives
    \begin{equation}
        G^{(3)}_i(t)\leq C_{2,2}'\gamma_i\frac{1}{|\ln \varepsilon|^{1+b}}\leq C_{2,2}''\frac{1}{|\ln \varepsilon|^{2+2b}}.
    \end{equation}

   It follows immediately from a combination of  the above estimates that
    \begin{equation}
       \sum_iE_i(t) \leq\frac{|X_0|}{2\pi h}\sum_i a_i^2\frac{1}{|\ln \varepsilon|^{1+2b}}-\sum_iG^{(2)}_i(t)+\overline{C}_{2,2}\frac{\ln |\ln \varepsilon|}{|\ln \varepsilon|^{2+2b}},
    \end{equation}
    which,  together with  Lemma \ref{6}, yields
    \begin{equation*}
        \sum_iG^{(2)}_i(t)\leq(C_1+\overline{C}_{2,2})\frac{\ln |\ln \varepsilon|}{|\ln \varepsilon|^{2+2b}}.
    \end{equation*}
    Since $\omega_{i,\ep}(0)$ is assumed to have a definite sign, we infer from the above inequality that
    \begin{equation*}
        \mathcal{G}_i(t)\leq C_2\frac{\ln |\ln \varepsilon|}{|\ln \varepsilon|^{2+2b}}.
    \end{equation*}
    This completes the proof of \eqref{7}.
\end{proof}

\subsection{Mass concentration and bound on the moment of inertia}
By using the previous proposition, we are able to show that the following \emph{ weak localization} result, which means that the mass of each vortex helices is concentrated in a disk with vanishing radius.
\begin{proposition}\label{11}
    For $\varepsilon\in(0,\varepsilon_0)$ and $t\in [0,T_\varepsilon]$,  there exist two positive constants $C_3$ and $C_4$ depending on $|x_0|,a_i,\overline{d},h, \rho$  and points $q^{i,\varepsilon}(t)\in \mathbb{R}^2$,  such that if $R>\exp(C_3\ln |\ln \varepsilon|)$, then
    \begin{equation}\label{10}
         \int_{B\left(q^{i,\varepsilon}(t),\frac{\varepsilon}{|\ln \varepsilon|^{(1+b)/2}}R\right)}|\omega_{i,\varepsilon}(x,t)|dx\geq |\gamma_i|-\frac{C_4\ln |\ln \varepsilon|}{|\ln \varepsilon|^{1+b}\ln  R}.
    \end{equation}
\end{proposition}
\begin{proof}
    In what follows we omit the explicit dependence on $i$ and $t$ by introducing the shorten notation
    \begin{equation*}
        \omega(x)=\frac{\gamma_i}{|\gamma_i|}\omega_{i,\varepsilon}(x,t)=|\omega_{i,\varepsilon}(x,t)|,\quad \gamma'=|\gamma_i|,\quad a=|a_i|.
    \end{equation*}
    In view that $\int\omega(x)dx=\gamma'$, there exist a point $x^*_1$ and a constant $L_1>1$ satisfying
    \begin{equation*}
        M_1:=\int_{x_1<x^*_1-\frac{\varepsilon}{|\ln \varepsilon|^{(1+b)/2}}L_1}\omega(x)dx\leq\frac{\gamma'}{2},\quad M_3:=\int_{x_1>x^*_1+\frac{\varepsilon}{|\ln \varepsilon|^{(1+b)/2}}L_1}\omega(x)dx\leq\frac{\gamma'}{2}.
    \end{equation*}
    Define
    \begin{equation*}
        M_2:=\int_{|x_1-x^*_1|\leq\frac{\varepsilon}{|\ln \varepsilon|^{(1+b)/2}}L_1}\omega(x)dx.
    \end{equation*}
    By using estimate (\ref{7}), we conclude that
    \begin{equation*}
        M_1M_3\ln 2L_1\leq C_2\frac{\ln |\ln \varepsilon|}{|\ln \varepsilon|^{2+2b}}.
    \end{equation*}
    Therefore, one computes
    \begin{equation*}
        \begin{split}
            \gamma'^2&=(M_1+M_2+M_3)^2\leq\frac{\gamma'^2}{2}+2M_2^2+2M_2(M_1+M_3)+2M_1M_3\\
            &=\frac{\gamma'^2}{2}+2\gamma' M_2+2M_1M_3\leq \frac{\gamma'^2}{2}+2\gamma' M_2+2C_2\frac{\ln |\ln \varepsilon|}{|\ln \varepsilon|^{2+2b}\ln 2L_1},
        \end{split}
    \end{equation*}
    which immediately yields
    \begin{equation*}
        M_2\geq \frac{\gamma'}{4}-\frac{C_2}{\gamma'}\frac{\ln |\ln \varepsilon|}{|\ln \varepsilon|^{2+2b}\ln 2L_1}.
    \end{equation*}
    Then we obtain
    \begin{equation}\label{M2}
        M_2\geq \frac{\gamma'}{8},\quad \forall L_1\geq L^*_1:=\frac{1}{2}\exp\left(\frac{8C_2}{a^2}\ln |\ln \varepsilon|\right).
    \end{equation}
    Now, set
    \begin{equation*}
        M_1':=\int_{x_1<x^*_1-\frac{2\varepsilon}{|\ln \varepsilon|^{(1+b)/2}}L_1}\omega(x)dx,\quad M_3':=\int_{x_1>x^*_1+\frac{2\varepsilon}{|\ln \varepsilon|^{(1+b)/2}}L_1}\omega(x)dx,
    \end{equation*}
    Using again \eqref{7} and taking into account  \eqref{M2}, we  obtain
    \begin{equation*}
        \frac{\gamma'}{8}(M_1'+M_3')\ln  L_1\leq C_2\frac{\ln |\ln \varepsilon|}{|\ln \varepsilon|^{2+2b}}\quad \forall L_1\geq L_1^*,
    \end{equation*}
    which implies
    \begin{equation}\label{8}
        M_2':=\int_{|x_1-x^*_1|\leq\frac{2\varepsilon}{|\ln \varepsilon|^{(1+b)/2}}L_1}\omega(x)dx\geq \gamma'-\frac{8C_2\ln |\ln \varepsilon|}{\gamma'|\ln \varepsilon|^{2+2b}\ln  L_1}\quad \forall L_1\geq L_1^*.
    \end{equation}
    For  $L_1>L_1^*$, repeating the same argument in the $x_2$-direction for the function
    \begin{equation*}
        \tilde{\omega}(x)=\omega(x)\Pi \left(|x_1-x^*_1|\leq\frac{2\varepsilon}{|\ln \varepsilon|^{\frac{1+b}{2}}}L_1\right),
    \end{equation*}
    we can show  that there is a point $x^*_2$ such that
    \begin{equation}\label{9}
        \int_{|x_2-x^*_2|\leq\frac{2\varepsilon}{|\ln \varepsilon|^{(1+b)/2}}L_2}\tilde{\omega}(x)dx\geq M_2'-\frac{8C_2\ln |\ln \varepsilon|}{M_2'|\ln \varepsilon|^{2+2b}\ln  L_2}\quad \forall L_2\geq L_2^*,
    \end{equation}
    where the constant
    \begin{equation*}
        L^*_2:=\frac{1}{2}\exp\left(\frac{4C_2}{(M_2')^2}\frac{\ln |\ln \varepsilon|}{|\ln \varepsilon|^{2+2b}}\right)\geq \frac{1}{2}\exp\left(\frac{8C_2}{\gamma'^2}\frac{\ln |\ln \varepsilon|}{|\ln \varepsilon|^{2+2b}}\right).
    \end{equation*}
    Since  $M_2'\geq M_2\geq \frac{\gamma'}{8}$, we obtain
    \begin{equation*}
        L^*_2\leq \frac{1}{2}\exp\left(\frac{256C_2}{\gamma'^2}\frac{\ln |\ln \varepsilon|}{|\ln \varepsilon|^{2+2b}}\right).
    \end{equation*}
    Therefore, letting $x^*=(x^*_1,x^*_2)$ and choosing
    \begin{equation}
        L_1=L_2=L>\frac{1}{2}\exp\left(\frac{256C_2}{\gamma'^2}\frac{\ln |\ln \varepsilon|}{|\ln \varepsilon|^{2+2b}}\right),
    \end{equation}
    from (\ref{8}) and (\ref{9}) we obtain
    \begin{equation*}
        \begin{split}
            \int_{B\left(x^*,\frac{2\sqrt{2}\varepsilon L}{|\ln \varepsilon|^{(1+b)/2}}\right)}\omega(x)dx&\geq \int_{|x_2-x^*_2|\leq\frac{2\varepsilon}{|\ln \varepsilon|^{(1+b)/2}}L}\tilde{\omega}(x)dx\\
            &\geq \gamma'-\frac{8C_2\ln |\ln \varepsilon|}{\gamma'|\ln \varepsilon|^{2+2b}\ln L_1}-\frac{8C_2\ln |\ln \varepsilon|}{M_2'|\ln \varepsilon|^{2+2b}\ln L_2}\\
            &\geq \gamma'-\frac{72C_2\ln |\ln \varepsilon|}{a|\ln \varepsilon|^{1+b}\ln L}.\\
        \end{split}
    \end{equation*}
    Choosing $R=2\sqrt{2}L$, $q^{i,\varepsilon}(t)=x^*$ and $C_3$, $C_4$ suitably, we derive the desired  inequality (\ref{10}) and thus complete the proof.
\end{proof}

Define the center of vorticity of the $i$-th vortex helix by
\begin{equation}\label{center}
    B^{i,\varepsilon}(t):=\frac{1}{\gamma_i}\int x\omega_{i,\varepsilon}(x,t)dx,
\end{equation}
and the corresponding moment of inertia by
\begin{equation}\label{moment}
    J_{i,\varepsilon}(t):=\int |x-B^{i,\varepsilon}(t)|^2|\omega_{i,\varepsilon}(x,t)|dx.
\end{equation}
\begin{lemma}\label{thm-moment}
    Given $\sigma\in(0,1)$ there exists $\varepsilon_\sigma\in(0,\varepsilon_0)$ such that
    \begin{equation}\label{15}
        J_{i,\varepsilon}(t)\leq \frac{1}{|\ln \varepsilon|^{\sigma+3+b}}\quad \forall t\in[0,T_\varepsilon]\quad \forall \varepsilon\in (0,\varepsilon_\sigma).
    \end{equation}
\end{lemma}
\begin{proof}
    Without loss of generality we assume hereafter $\omega_{i,\varepsilon}(t)\geq 0$ and hence $\gamma_i>0$ and
    \begin{equation*}
        J_{i,\varepsilon}(t)=\int |x-B^{i,\varepsilon}(t)|^2\omega_{i,\varepsilon}(x,t)dx.
    \end{equation*}
    Given $\sigma\in(0,1)$, we choose $\tilde{\sigma}\in(\sigma,1)$ and let
    \begin{equation}
        B_{i,\varepsilon}(t):=B(q^{i,\varepsilon}(t), \varepsilon |\ln \varepsilon|^{-(1+b)/2}R_\varepsilon),\quad R_\varepsilon:=\exp(|\ln \varepsilon|^{\tilde{\sigma}}).
    \end{equation}
     It is clear that for sufficiently small $\ep$, there holds  $R_\varepsilon>\exp(C_3\ln |\ln \varepsilon|)$. Then,  we can apply (\ref{10}) in Proposition \ref{11} with $R=R_\varepsilon$   to ensure that 
     \begin{equation}\label{12}
         \int_{B_{i,\varepsilon}(t)}\omega_{i,\varepsilon}(y,t)dy\geq \gamma_i-\frac{C_4\ln |\ln \varepsilon|}{|\ln \varepsilon|^{\tilde{\sigma}+1+b}}\quad \forall t\in[0,T_\varepsilon].
     \end{equation}
    By the definitions of $B^{i,\varepsilon}(t)$ and $J_{i,\varepsilon}(t)$, we have
    \begin{equation*}
        \begin{split}
           J_{i,\varepsilon}(t)&=\min_{q\in \mathbb{R}^2} \int |x-q|^2\omega_{i,\varepsilon}(x,t)dx\leq \int |x-q^{i,\varepsilon}(t)|^2\omega_{i,\varepsilon}(x,t)dx\\
           &=\int_{B_{i,\varepsilon}(t)} |x-q^{i,\varepsilon}(t)|^2\omega_{i,\varepsilon}(x,t)dx+\int_{B_{i,\varepsilon}(t)^\complement} |x-q^{i,\varepsilon}(t)|^2\omega_{i,\varepsilon}(x,t)dx\\
           &\leq \gamma_i\left(\frac{\varepsilon}{|\ln \varepsilon|^{\frac{1+b}{2}}}R_\varepsilon\right)^2+\frac{C_4\ln |\ln \varepsilon|}{|\ln \varepsilon|^{\tilde{\sigma}+1+b}}\max_{x\in \Lambda_{i,\varepsilon}(t)}|x-q^{i,\varepsilon}(t)|^2.
        \end{split}
    \end{equation*}
    It follows from (\ref{12}) that $B_{i,\varepsilon}(t)\cap \Lambda_{i,\varepsilon}(t)\neq\emptyset$, then one has
    \begin{equation*}
        |q^{i,\varepsilon}(t)|\leq|x|+|x-q^{i,\varepsilon}(t)|\leq\frac{\overline{d}+\rho}{|\ln \varepsilon|}+\frac{\varepsilon R_\varepsilon}{|\ln \varepsilon|}.
    \end{equation*}
    Therefore, we deduce from the above estimates that for sufficiently small $\ep>0$,
    \begin{equation*}
        J_{i,\varepsilon}(t)\leq \frac{C_5\ln |\ln \varepsilon|}{|\ln \varepsilon|^{\tilde{\sigma}+3+b}}\leq \frac{1}{|\ln \varepsilon|^{\sigma+3+b}},
    \end{equation*}
which deduces \eqref{15} and the proof is thus finished.
\end{proof}

\subsection{Improved estimates via an iterative procedure}
It can be seen that the time $T_\ep$ in \eqref{defTep} may vanish as $\ep\to0$. In this subsection, we aim to find a uniform lower bound $T'_\rho$ so that  $T_\ep\geq   T'_\rho$ for any $\ep>0$ sufficiently small. This is achieved by an iterative procedure and a continuous argument. Being different with existing literature, in our later proof, we need to introduce a coordinate-transformed cut-off function $W_{R,\zeta}\big(DT(x_0)(x-B^{i,\varepsilon}(t))\big)$ to get sufficient cancellations.

For a matrix $Q=\begin{pmatrix}
	a&b\\c&d
\end{pmatrix}$, in the following, we denote  $Q^*:=\begin{pmatrix}
	d&-c\\-b&a
\end{pmatrix}$, so that $(Q x)^\perp=Q^* x^\perp$.  According to the decomposition in \eqref{decom DG}, we divide the velocity into four parts as follows:
\begin{equation}\label{decom v}
	v(x,t)=v^i_K(x,t)+v^i_L(x,t)+v^i_S(x,t)+F^i(x,t),
\end{equation}
where
\begin{equation*}
	v^i_K(x,t)=\left(\int H(x,y)DT(x)\frac{T(x)-T(y)}{|T(x)-T(y)|^2}\omega_{i,\varepsilon}(y,t)dy\right)^\perp,
\end{equation*}
\begin{equation*}
	v^i_L(x,t)=\frac{x^\perp}{2|X|^2}\int G_K(x,y)\omega_{i,\varepsilon}(y,t)dy,
\end{equation*}
\begin{equation*}
	v^i_S(x,t)=\int\nabla^\perp S_{K} (x,y)\omega_{i,\varepsilon}(y,t)dy,
\end{equation*}
\begin{equation}
F^i(x,t)=\sum_{i\neq j}\int\nabla^\perp \mathcal{G}_K (x,y)\omega_{j,\ep}(y,t)dy.
\end{equation}
 Since $ S_{K} (x,y)$ is bounded in $W^{1,\infty}$ on $\Lambda_{i,\varepsilon}(t)$ for every $t\in[0,T_\varepsilon]$, there is a constant $C_S=C_S(|x_0|,\overline{d},h)>0$ such that
\begin{equation}\label{estimate vs}
	|v^i_S(x,t)|\leq C_S|\gamma_i|\leq C_S\overline{\gamma}.
\end{equation}

\begin{lemma}\label{estimate vf}
    There exists a constant $C_F=C_F(|x_0|,a_i,h)$ such that for every $\varepsilon\in(0,\varepsilon_0)$ small enough, for every $i\in \{1,...,N\}$, for every $x,y\in \Lambda_{i,\varepsilon}(t)$ and for every $t\in [0,T_\varepsilon]$,
    \begin{equation}\label{14}
        \left\{
\begin{array}{ll}
|F^i(x,t)-F^i(y,t)|\leq \frac{C_F}{\rho^2}\frac{a}{|ln\varepsilon|^{b-1}}|x-y|,\\
|F^i(x,t)|\leq \frac{C_F}{\rho}\frac{a}{|\ln \varepsilon|^b},\\
\mathrm{div} F^i(x,t)=0.
\end{array}
\right.
    \end{equation}
\end{lemma}
\begin{proof}
    Notice that when  $x$ and $y$ are close, the singularity of  Green's function $\mathcal{G}_K(x,y)  $
    is logarithmic. By using \eqref{3-4}, it can be verified that
    \begin{align*}
    	|\nabla \mathcal{G}_K(x,y)|\leq C \frac{|\ln\ep|}{\rho},\ \ |\nabla^2 \mathcal{G}_K(x,y)|\leq C \frac{|\ln\ep|^2}{\rho^2}, \ \ \forall \  x\in \Lambda_{i,\varepsilon}(t), y\in \Lambda_{j,\varepsilon}(t)
    \end{align*}
    for some $C$ depending on $|x_0|, \bar{d}, h$. A rigorous verification of the above estimates can also be obtained by using the refined decomposition of Green's function in \cite{CW3}. Then the desired conclusion follows immediately from the above estimates.
\end{proof}
Denote
    \begin{equation}\label{2C}
        c_0=\tau(r_0^2), \quad C_0=\tau(r_0^2)\left(1+\frac{r_0^2}{h^2+h\sqrt{h^2+r_0^2}}\right)
    \end{equation}
and $c_0'=1/C_0$, $C_0'=1/c_0$.  By the explicit expression of $DT(x_0)$ and $DT(x_0)^{-1}$, for every $z\in\mathbb{R}^2$, there hold
    \begin{equation*}
     c_0|z|\leq|DT(x_0)z|\leq C_0|z|,\quad c_0'|z|\leq |DT(x_0)^{-1}z|\leq C_0'|z|.
    \end{equation*}
In particular, when $x_0=0$, we have $DT(x_0)=I_2$ and $C_0=c_0=1$.

Before we proceed to the next step of estimation, we need to recall a lemma due to \cite{DLM} is needed. To this end
denote $D^\complement =\{x\in \mathbb{R}^2\backslash D\}$ for any given set $D\subset\mathbb{R}^2$.
\begin{lemma}[Lemma 4.3 in \cite{DLM}]\label{13}
    Let $\alpha\in C^1(\mathbb R^2\times[0,T_\varepsilon],\mathbb{R})$. Then for every $t\in[0,T_\varepsilon]$ and $i=1,...,N$,
    \begin{equation*}
        \frac{d}{dt}\int\alpha(x,t)\omega_{i,\varepsilon}(x,t)dx=\int\partial_t\alpha(x,t)\omega_{i,\varepsilon}(x,t)dx+\frac{1}{|\ln\varepsilon|^{1-b}}\int\nabla\alpha(x,t)\cdot v(x,t)\omega_{i,\varepsilon}(x,t)dx.
    \end{equation*}
\end{lemma}
\begin{proposition}\label{20}
    Let
    \begin{equation*}
        m^i_t(R):=\int_{B(B^{i,\varepsilon}(t),\frac{R}{|\ln \varepsilon|})^\complement}|\omega_{i,\varepsilon}(x,t)|dx
    \end{equation*}
    denote the amount of vorticity of the i-th helix outside the disk $B(B^{i,\varepsilon}(t),\frac{R}{|\ln \varepsilon|})$ at time $t$. Then, given $R>0$, for each $l\in\mathbb{R}$ there is $\tilde{T}_l\in(0,T]$ such that  for any $\varepsilon\in(0,\varepsilon_0)$ sufficiently small we have
    \begin{equation}\label{19}
        m^i_t(R)\leq\frac{1}{|\ln \varepsilon|^{1+b+l}}\quad \forall t\in[0,\tilde{T}_l\wedge T_\varepsilon]\quad\forall i=1,...,N.
    \end{equation}

\end{proposition}
\begin{proof}
    Denote $\mathcal{A}_R=\left[B^{i,\varepsilon}(t)+DT(x_0)^{-1}B(0,\frac{R}{|\ln\varepsilon|})\right]$, and
    \begin{equation*}
        \tilde{m}^i_t(R):=\int_{\mathcal{A}_R^\complement}|\omega_{i,\varepsilon}(x,t)|dx.
    \end{equation*}
Then we have $m^i_t(C_0'R)\leq \tilde{m}^i_t(R)\leq m^i_t(c_0'R)$.

Without loss of generality, we may assume that $\omega_{i,\varepsilon}(t)\geq0$ and hence   $\gamma_i=|\gamma_i|$ and
    \begin{equation*}
        m^i_t(R):=\int_{B(B^{i,\varepsilon}(t),\frac{R}{|\ln \varepsilon|})^\complement}\omega_{i,\varepsilon}(x,t)dx.
    \end{equation*}
    Let $\zeta,R$  be two positive parameters to be fixed later such that  $R\geq 2\zeta$. Take a  non-negative smooth cut-off function $ W_{R,\zeta} :  \mathbb{R}^2 \mapsto \mathbb R$  depending only on $|x|$ satisfying
    \begin{equation}
        W_{R,\zeta}(x)=\begin{cases}
1 \quad &if\quad|x|\leq\frac{R}{|\ln \varepsilon|},\\
0\quad&if\quad|x|\geq\frac{R+\zeta}{|\ln \varepsilon|}.
\end{cases}
    \end{equation}
Then,  for some $C_W>0$, it holds that
    \begin{equation}\label{4-7}
    	\begin{cases}
        |\nabla W_{R,\zeta}(x)|=\left|\eta_\zeta(|x|)\frac{x}{|x|}\right|<\frac{|\ln \varepsilon|C_W}{\zeta},&\\
        |\nabla W_{R,\zeta}(x)-\nabla W_{R,\zeta}(x')|<\frac{|\ln \varepsilon|^2C_W}{\zeta^2}|x-x'|,&\\
        \end{cases}
    \end{equation}
where by the definition of  $W_{R,\zeta}$, $\eta_\zeta$ satisfies
    \begin{equation*}
    	\begin{cases}
         \eta_{\zeta}(s)=0, &\text{if}\ \,s\in (0,\frac{R}{|\ln \varepsilon |})\bigcup (\frac{R+\zeta}{|\ln \varepsilon |},+\infty ),\\
         0\leq\eta_{\zeta}(s)\leq \frac{|\ln \varepsilon|C_W}{\zeta}, &\text{if}\ \,s\in [\frac{R}{|\ln \varepsilon |},\frac{R+\zeta}{|\ln \varepsilon |}].\\
        \end{cases}
    \end{equation*}

    It can be seen that the quantity
    \begin{equation}
        \mu_t(R,\zeta)=\int[1-W_{R,\zeta}(DT(x_0)(x-B^{i,\varepsilon}(t)))]\omega_{i,\varepsilon}(x,t)dx
    \end{equation}
    is a mollified version of $m^i_t$, satisfying
    \begin{equation}
        \mu_t(R,\zeta)\leq \tilde{m}^i_t(R)\leq \mu_t(R-\zeta,\zeta).
    \end{equation}

    Next, we study the variation in time of $\mu_t(R,\zeta)$ by applying Lemma \ref{13} with test function $\alpha(x,t)=1-W_{R,\zeta}(x-B^{i,\varepsilon}(t))$, getting
    \begin{equation*}
        \frac{d}{dt}\mu_t(R,\zeta)=-\frac{1}{|\ln\varepsilon|^{1-b}}\int\nabla W_{R,\zeta}(DT(x_0)(x-B^{i,\varepsilon}(t)))\cdot DT(x_0)[v(x,t)-\dot{B}^{i,\varepsilon}(t)] \omega_{i,\varepsilon}(x,t)dx.
    \end{equation*}
    Applying again Lemma \ref{13} with test function $\alpha(x,t)=x$, we get
    \begin{equation}
       \dot{B}^{i,\varepsilon}(t)=\frac{1}{\gamma_i} \int v(y,t)\omega_{i,\varepsilon}(y,t)dy.
    \end{equation}
    By using the decomposition of velocity $v$ in \eqref{decom v},  we conclude that
    \begin{equation}
        \frac{d}{dt}\mu_t(R,\zeta)=-\frac{1}{|\ln\varepsilon|^{1-b}}(A_1+A_2+A_3),
    \end{equation}
    where
    \begin{equation*}
        \begin{split}
            A_1&=\int\nabla W_{R,\zeta}(DT(x_0)(x-B^{i,\varepsilon}(t)))\cdot DT(x_0) v^i_S(x,t)\omega_{i,\varepsilon}(x,t)dx\\
            &-\frac{1}{\gamma_i}\int\nabla W_{R,\zeta}(DT(x_0)(x-B^{i,\varepsilon}(t)))\cdot DT(x_0)\int v^i_S(y,t)\omega_{i,\varepsilon}(y,t)\omega_{i,\varepsilon}(x,t)dydx\\
            &-\int\nabla W_{R,\zeta}(DT(x_0)(x-B^{i,\varepsilon}(t)))\cdot DT(x_0) F^i(x,t)\omega_{i,\varepsilon}(x,t)dx\\
            &-\frac{1}{\gamma_i}\int\nabla W_{R,\zeta}(DT(x_0)(x-B^{i,\varepsilon}(t)))\cdot DT(x_0)\int F^i(y,t)\omega_{i,\varepsilon}(y,t)\omega_{i,\varepsilon}(x,t)dydx,\\
            A_2&=\frac{1}{\gamma_i}\int\nabla W_{R,\zeta}(DT(x_0)(x-B^{i,\varepsilon}(t)))\cdot DT(x_0)\int(v^i_K(x,t)-v^i_K(y,t))\omega_{i,\varepsilon}(y,t)\omega_{i,\varepsilon}(x,t)dydx,\\
            A_3&=\frac{1}{\gamma_i}\int\nabla W_{R,\zeta}(DT(x_0)(x-B^{i,\varepsilon}(t)))\cdot DT(x_0)\int(v^i_L(x,t)-v^i_L(y,t))\omega_{i,\varepsilon}(y,t)\omega_{i,\varepsilon}(x,t)dydx.
        \end{split}
    \end{equation*}

    We further decompose $A_1$ as $A_1=A^1_1+A^2_1$ with
       \begin{align*}
    		A^1_1&=\int\nabla W_{R,\zeta}(DT(x_0)(x-B^{i,\varepsilon}(t)))\cdot DT(x_0) v^i_S(x,t)\omega_{i,\varepsilon}(x,t)dx\\
    		&-\frac{1}{\gamma_i}\int\nabla W_{R,\zeta}(DT(x_0)(x-B^{i,\varepsilon}(t)))\cdot DT(x_0)\int v^i_S(y,t)\omega_{i,\varepsilon}(y,t)\omega_{i,\varepsilon}(x,t)dydx,\\
    		A^2_1=&-\int\nabla W_{R,\zeta}(DT(x_0)(x-B^{i,\varepsilon}(t)))\cdot DT(x_0) F^i(x,t)\omega_{i,\varepsilon}(x,t)dx\\
    		&-\frac{1}{\gamma_i}\int\nabla W_{R,\zeta}(DT(x_0)(x-B^{i,\varepsilon}(t)))\cdot DT(x_0)\int F^i(y,t)\omega_{i,\varepsilon}(y,t)\omega_{i,\varepsilon}(x,t)dydx.
    \end{align*}
    For $A^1_1$, by using \eqref{estimate vs} and \eqref{4-7}, we have
    \begin{equation}\label{a11}
    \begin{split}
        |A^1_1|&\leq \frac{|\ln \varepsilon|C_W}{\zeta}\int_{\mathcal{A}_R^\complement}C_S\overline{\gamma}\omega_{i,\varepsilon}(x,t)dx+\frac{|\ln \varepsilon|C_W}{\zeta}\cdot C_S\overline{\gamma}\frac{1}{\gamma_i}\int\omega_{i,\varepsilon}(y,t)dy\int_{\mathcal{A}_R^\complement}\omega_{i,\varepsilon}(x,t)dx \\
        &\leq \frac{C_W'}{\zeta|\ln \varepsilon|^b}\tilde{m}^i_t(R).
    \end{split}
    \end{equation}
    We next use inequalities (\ref{15}) and (\ref{14}) to estimate  $A^2_1$.
    \begin{equation}\label{a12}
        \begin{split}
            |A^2_1|&=\left|\frac{1}{\gamma_i}\int\nabla W_{R,\zeta}(DT(x_0)(x-B^{i,\varepsilon}(t)))\cdot DT(x_0)\int(F^i(x,t)-F^i(y,t))\omega_{i,\varepsilon}(y,t)\omega_{i,\varepsilon}(x,t)dydx\right|\\
            &\leq\frac{|\ln \varepsilon|^{2-b} C_WC_F a}{\gamma_i \zeta\rho^2}\int_{\mathcal{A}_{R+\zeta}\backslash \mathcal{A}_R}\omega_{i,\varepsilon}(x,t)\int_{\mathcal{A}_R}\omega_{i,\varepsilon}(y,t)|x-y|dydx\\
            &+\frac{2|\ln\varepsilon|^{1-b} C_WC_F a  }{\gamma_i \zeta\rho}\int_{\mathcal{A}_R^\complement}\omega_{i,\varepsilon}(x,t)dx\int_{\mathcal{A}_R^\complement}\omega_{i,\varepsilon}(y,t)dy\\
            &\leq\frac{|\ln\varepsilon|^{1-b}C_WC_F a}{c_0'\zeta\rho^2}(2R+\zeta)  \tilde{m}^i_t(R)+\frac{2|\ln\varepsilon|^{1-b}C_WC_F a}{ a_i c_0'^2 R^2\zeta\rho}|\ln \varepsilon|^{3+b}J_{i,\varepsilon}(t)\tilde{m}^i_t(R)\\
            &\leq\frac{3|\ln\varepsilon|^{1-b}RC_WC_Fa}{c_0'\zeta\rho^2}\tilde{m}^i_t(R)+\frac{2|\ln\varepsilon|^{1-b}C_WC_Fa}{a_ic_0'^2R^2\zeta\rho|\ln \varepsilon|^\sigma}\tilde{m}^i_t(R).
        \end{split}
    \end{equation}Here, in the above calculations,  we have used   Chebyshev's inequality to conclude
    \begin{equation}\label{Cheb}
    	\int_{\mathcal{A}_R^\complement}\omega_{i,\varepsilon}(y,t)dy\leq \frac{|\ln\ep|^2}{c_0'^2R^2} J_{i, \ep}(t).
    \end{equation}

    For $A_2$, we first consider the term
    $$A^1_2:= \frac{1}{|\ln\varepsilon|^{1-b}\gamma_i}\int\nabla W_{R,\zeta}(DT(x_0)(x-B^{i,\varepsilon}(t)))\cdot DT(x_0)\int v^i_K(x,t) \omega_{i,\varepsilon}(y,t)\omega_{i,\varepsilon}(x,t)dydx.$$
    By the definition of $v_K^i$ in \eqref{decom v} and symmetry properties, we obtain
    \begin{align*}
            A^1_2=&\frac{1}{|\ln\varepsilon|^{1-b}}\iint\eta_\zeta(|DT(x_0)(x-B^{i,\varepsilon}(t))|)\frac{DT(x_0)(x-B^{i,\varepsilon}(t))}{|DT(x_0)(x-B^{i,\varepsilon}(t))|}\cdot DT(x_0)\\&\left(H(x,z)DT(x)^*\frac{(T(x)-T(z))^\perp}{|T(x)-T(z)|^2}\right)\omega_{i,\varepsilon}(z,t)\omega_{i,\varepsilon}(x,t)dzdx\\
            =&\frac{1}{2|\ln\varepsilon|^{1-b}}\iint DT(x_0)\left[\nabla W_{R,\zeta}(DT(x_0)(x-B^{i,\varepsilon}(t)))DT(x)^*-\nabla W_{R,\zeta}(DT(x_0)(z-B^{i,\varepsilon}(t)))DT(z)^*\right]\\
            &H(x,z)\frac{(T(x)-T(z))^\perp}{|T(x)-T(z)|^2}\omega_{i,\varepsilon}(z,t)\omega_{i,\varepsilon}(x,t)dzdx\\
            =&\frac{1}{2|\ln\varepsilon|^{1-b}}\iint f(x,z,t)dzdx,
    \end{align*}
    where $f(x,z,t)=f(z,x,t)$ and $f(x,z,t)=0$, $\forall (x,z)\in \mathcal{A}_R\times\mathcal{A}_R $.
    Therefore, we decompose the integral into three parts as follows:
    \begin{equation*}
    \begin{array}{ll}
         \iint f(x,z,t)dzdx=2\iint_{\mathcal{A}_R^\complement\times \mathcal{A}_{R/2}}f(x,z,t)dzdx
         +2\iint_{\mathcal{A}_R^\complement\times \mathcal{A}_{R/2}^\complement}f(x,z,t)dzdx&\\
         \qquad\qquad\qquad\qquad\quad -\iint_{\mathcal{A}_R\times\mathcal{A}_R}f(x,z,t)dzdx.
   \end{array}
    \end{equation*}
    We infer from \eqref{4-7} that
    \begin{equation}
        \left|2\iint_{\mathcal{A}_R^\complement\times \mathcal{A}_{R/2}}f(x,z,t)dzdx\right|\\
        \leq\frac{C_W|\ln\varepsilon|}{2h}\tilde{m}^i_t(R)\sup_{x\in \mathcal{A}_R^\complement}|\tilde{H}_1(x)|,
    \end{equation}
where
   \begin{align*}
        \tilde{H}_1(x)&=\frac{DT(x_0)(x-B^{i,\varepsilon}(t))}{|DT(x_0)(x-B^{i,\varepsilon}(t))|}  \cdot\int_{\mathcal{A}_{R/2}} H(x,y)DT(x_0)DT(x)^*\frac{(T(x)-T(y))^\perp}{|T(x)-T(y)|^2}\omega_{i,\varepsilon}(y,t)dy\\
        &=\frac{DT(x_0)(x-B^{i,\varepsilon}(t))}{|DT(x_0)(x-B^{i,\varepsilon}(t))|}\\
        &\qquad  \cdot\int_{\mathcal{A}_{R/2}} H(x_0,x_0)DT(x_0)DT(x_0)^*\frac{(DT(x_0)(x-y))^\perp}{|DT(x_0)(x-y)|^2}\omega_{i,\varepsilon}(y,t)dy+O\left(\frac{1}{|\ln \varepsilon|^{1+b}}\right).
\end{align*}
    Here we have used the fact $|x_0-y|\leq  \frac{\overline{d}+\rho}{|\ln \varepsilon|}$ for any $y\in \Lambda_{i,\varepsilon}(t)$ and any $t\in[0,T_\varepsilon)$.

    Denote $z'=z-B^{i,\varepsilon}(t)$ for a point $z\in \mathbb R^2$. Also denote $\tilde {z}'=DT(x_0)z'$ and $\overline{\omega}_{i,\varepsilon}(\tilde{y}',t)=\omega_{i,\varepsilon}(y,t)$. Notice that $DT(x_0)DT(x_0)^*$ equals to a constant multiplied by the identity matrix. Then the above identity can be rewritten as follows:
    \begin{equation}
        \tilde{H}_1(x)=C_{K,1}\int_{|\tilde{y}'|\leq\frac{R}{2|\ln \varepsilon|}}\frac{-\tilde x'\cdot \tilde y'^\perp}{|\tilde x'||\tilde x'-\tilde y' |^2}\overline{\omega}_{i,\varepsilon}(\tilde{y}',t)d\tilde{y}'+O\left(\frac{1}{|\ln \varepsilon|^{1+b}}\right).
    \end{equation}
    Using the identity $\int \tilde y' \overline{\omega}_{i,\varepsilon}(\tilde{y}',t)d\tilde{y}'=0$ due to the definition of $B^{i,\varepsilon}(t)$, we further decompose $\tilde{H}_1(x)$ as follows:
    \begin{equation}
        \tilde{H}_1(x)=\tilde{H}_1'(x)-\tilde{H}_1''(x)+O\left(\frac{1}{|\ln \varepsilon|^{1+b}}\right),
    \end{equation}
    where
    \begin{equation*}
        \tilde{H}_1'(x)=-C_{K,1}\int_{|\tilde{y}'|\leq\frac{R}{2|\ln \varepsilon|}}\frac{(\tilde x'\cdot \tilde y'^\perp) \tilde y'\cdot(2\tilde  x'-\tilde y')}{|\tilde x'||\tilde x'-\tilde y' |^2|\tilde x'|^2}\overline{\omega}_{i,\varepsilon}(\tilde{y}',t)d\tilde{y}',
    \end{equation*}
    \begin{equation*}
        \tilde{H}''_1(x)=-C_{K,1}\int_{|\tilde{y}'|\geq\frac{R}{2|\ln \varepsilon|}}\frac{\tilde x_1'\cdot \tilde y'^\perp}{|\tilde x_1'|^3}\overline{\omega}_{i,\varepsilon}(\tilde{y}',t)d\tilde{y}'.
    \end{equation*}
    Clearly,  one has  $|\tilde x_1'|\geq\frac{R}{|\ln \varepsilon|}$, and  for $|\tilde{y}'|\leq\frac{R}{2|\ln \varepsilon|}$, there hold  $|2\tilde x_1'-\tilde y'|\leq |\tilde x_1'-\tilde y'|+|\tilde x_1'|$  and $|\tilde x_1'-\tilde y'|\geq\frac{R}{2|\ln \varepsilon|}$. Therefore, we obtain
    \begin{equation}
        \begin{split}
            |\tilde{H}'_1(x)|\leq  C'_{K,1}\frac{|\ln \varepsilon|^3}{R^3}\int_{|\tilde{y}'|\leq\frac{R}{2|\ln \varepsilon|}}|\tilde y'|^2\overline{\omega}_{i,\varepsilon}(\tilde{y}',t)d\tilde{y}' \leq C'_{K,1}\frac{|\ln \varepsilon|^3}{R^3}J_{i,\varepsilon}(t).
        \end{split}
    \end{equation}
    It follows easily from   Chebyshev's inequality that
    \begin{equation}
        \begin{split}
            |\tilde{H}''_1(x)| \leq\tilde{C}_H\frac{1}{|\tilde x_1'|^2}\int_{|\tilde{y}'|\geq\frac{R}{2|\ln \varepsilon|}}|\tilde y'|\overline{\omega}_{i,\varepsilon}(\tilde{y}',t)d\tilde{y}' \leq\tilde{C}_{H,1}\frac{|\ln \varepsilon|^3}{R^3}J_{i,\varepsilon}(t).
        \end{split}
    \end{equation}
    Then by  (\ref{15}), we have
    \begin{equation}
        |\tilde{H}_1(x)|\leq C'''_{K,1}\left(\frac{1}{R^3|\ln \varepsilon|^{b+\sigma}}+\frac{1}{|\ln \varepsilon|^{1+b}}\right).
    \end{equation}
    On the other hand, \eqref{4-7} yields
    \begin{equation*}
    \begin{split}
        &|DT(x_0)\nabla W_{R,\zeta}(DT(x_0)(x-B^{i,\varepsilon}(t)))DT(x)^*-DT(x_0)\nabla W_{R,\zeta}(DT(x_0)(z-B^{i,\varepsilon}(t)))DT(z)^*|\\\leq &C''_K\left(\frac{|\ln \varepsilon|^2}{\zeta ^2}+\frac{|\ln \varepsilon|}{\zeta }\right)|x-y|.
    \end{split}
    \end{equation*}
    Combining  (\ref{15}) and \eqref{Cheb} (with $R$ replaced by $R/2$), we get
    \begin{equation}\label{16}
        \begin{split}
          &\left|\iint_{\mathcal{A}_R^\complement\times \mathcal{A}_{R/2}^\complement}f(x,z,t)dzdx \right|+\left|\iint_{\mathcal{A}_R^\complement\times\mathcal{A}_R^\complement}f(x,z,t)dzdx\right| \\
          \leq&C \left(\frac{|\ln \varepsilon|^2}{\zeta ^2}+\frac{|\ln \varepsilon|}{\zeta }\right) \int_{\mathcal{A}_{R/2}^\complement} \omega_{i,\varepsilon}(x,t)dx\cdot \tilde{m}^i_t(R)\\
          \leq& C \left(\frac{|\ln \varepsilon|^2}{\zeta^2}+\frac{|\ln \varepsilon|}{\zeta}\right)\frac{4|\ln \varepsilon|^2}{c_0'^2R^2}\frac{1}{|\ln \varepsilon|^{\sigma+3+b}}\tilde{m}^i_t(R)\\
          =&C'_K\left(\frac{1}{\zeta^2}+\frac{1}{\zeta|\ln \varepsilon|}\right)\frac{1}{R^2}\cdot \frac{1}{|\ln \varepsilon|^{\sigma+b-1}}\tilde{m}^i_t(R).
        \end{split}
    \end{equation}
    Now, we turn to the remaining term in $A_2$, which is given by
    \begin{align*}
    		A^2_2&:=\frac{1}{\gamma_i}\int\nabla W_{R,\zeta}(DT(x_0)(x-B^{i,\varepsilon}(t)))\cdot DT(x_0)\int(-v^i_K(y,t))\omega_{i,\varepsilon}(y,t)\omega_{i,\varepsilon}(x,t)dydx\\
    		&=-\frac{1}{\gamma_i} \iiint\nabla W_{R,\zeta}(DT(x_0)(x-B^{i,\varepsilon}(t)))\omega _{i,\varepsilon}(x,t) H(z,y)DT(x_0)\\
    		&\qquad \qquad \qquad \times DT(y)^*\frac{(T(z)-T(y))^\perp}{|T(z)-T(y)|^2}\omega_{i,\varepsilon}(y,t)\omega_{i,\varepsilon}(z,t)dxdydz.
    \end{align*}
    It follows immediately that
    \begin{equation*}
        \left|\int\nabla W_{R,\zeta}(DT(x_0)(x-B^{i,\varepsilon}(t)))\omega _{i,\varepsilon}(x,t)dx\right|\leq \frac{|\ln \varepsilon|C_W}{\zeta}\tilde{m}^i_t(R),
    \end{equation*}
    \begin{equation*}
        \begin{split}
           &\left| \iint H(z,y)DT(x_0)DT(y)^*\frac{(T(z)-T(y))^\perp}{|T(z)-T(y)|^2}\omega_{i,\varepsilon}(y,t)\omega_{i,\varepsilon}(z,t)dydz\right|\\
           =&\left|\frac{1}{2}\iint H(z,y)DT(x_0)(DT(y)^*-DT(z)^*)\frac{(T(z)-T(y))^\perp}{|T(z)-T(y)|^2}\omega_{i,\varepsilon}(y,t)\omega_{i,\varepsilon}(z,t)dydz\right|\\
           \leq&C''_K\gamma^2_i.
        \end{split}
    \end{equation*}
    Therefore, we obtain
    \begin{equation}
        |A^2_2|\leq\frac{C''_KC_Wa_i}{\zeta |\ln \varepsilon|^b}\tilde{m}^i_t(R).
    \end{equation}

    For the term $A_3$, we write
    \begin{equation*}
        A_3=A^1_3+A^2_3,
    \end{equation*}
    with
    \begin{equation*}
        A^1_3=\int \nabla W_{R,\zeta}(DT(x_0)(x-B^{i,\varepsilon}(t)))\cdot DT(x_0)\frac{x^\perp}{2|X|^2}\int H(x,y)\ln |T(x)-T(y)|\omega_{i,\varepsilon}(x,t)\omega_{i,\varepsilon}(y,t)dydx.
    \end{equation*}
    Notice that
    \begin{equation*}
        \left|\frac{x^\perp}{2|X|^2}\right|\leq \frac{1}{|X_0|}\leq\frac{1}{h}.
    \end{equation*}
    Letting ${r}_{\ep, i}=\sqrt{\frac{|\gamma_i|}{M\pi}} \ep$, by using $|T(x)-T(y)|\leq C|x-y|$ and the rearrangement inequality, we obtain
    \begin{equation}\label{4-18}
        \begin{split}
            &\left|\int \ln |T(x)-T(y)|\omega_{i,\varepsilon}(y,t)dy\right|\\
            \leq&\left|\int \ln \frac{|T(x)-T(y)|}{|x-y|}\omega_{i,\varepsilon}(y,t)dy\right|+\left|\int \ln |x-y|\omega_{i,\varepsilon(y,t)}dy\right|\\
            \leq&C  \gamma_i+\left|\frac{M}{\varepsilon^2}\int^{{r}_{\ep, i}i}_02\pi r\ln rdr\right|
            \leq \frac{ C_L  }{|\ln \varepsilon|^b},
        \end{split}
    \end{equation}
    from which we get
    \begin{equation}\label{a13}
        |A^1_3|\leq\frac{2 C_W|\ln \varepsilon|}{\zeta h}\frac{ C_0C_L}{|\ln \varepsilon|^b}\tilde{m}^i_t(R)=\frac{|\ln\varepsilon|^{1-b}C''_L}{\zeta}\tilde{m}^i_t(R).
    \end{equation}
    Next we estimate $A^2_3$, which is given by
    \begin{equation*}
    \begin{split}
        A^2_3=&-\frac{1}{\gamma_i}\int\nabla W_{R,\zeta}(DT(x_0)(x-B^{i,\varepsilon}(t)))\omega _{i,\varepsilon}(x,t)dx\cdot DT(x_0)\\&\iint\frac{y^\perp}{2|Y|^2}H(y,z)\ln |T(y)-T(z)|\omega _{i,\varepsilon}(y,t)\omega _{i,\varepsilon}(z,t)dydz.
        \end{split}
    \end{equation*}
    Notice that
    \begin{equation*}
        \left|\int\nabla W_{R,\zeta}(DT(x_0)(x-B^{i,\varepsilon}(t)))\omega _{i,\varepsilon}(x,t)dx\right|\leq \frac{|\ln \varepsilon|C_W}{\zeta}\tilde{m}^i_t(R).
    \end{equation*}
    Using \eqref{4-18}, we have
    \begin{equation*}
        \begin{split}
            &\left|DT(x_0)\iint\frac{y^\perp}{2|Y|^2}H(y,z)\ln |T(y)-T(z)|\omega _{i,\varepsilon}(y,t)\omega _{i,\varepsilon}(z,t)dydz\right|\\
            \leq&\gamma_i\frac{C }{h}\sup_{z\in \Lambda_{i,\varepsilon}(t)}\left|\int \ln |T(y)-T(z)|\omega_{i,\varepsilon}(y,t)dy\right|
            \leq \frac{C'_L a_i}{|\ln \varepsilon|^{1+2b}},
        \end{split}
    \end{equation*}
    which yields
    \begin{equation}\label{a32}
        |A^2_3|\leq\frac{|\ln\varepsilon|^{1-b}\overline{C}_L}{\zeta}\tilde{m}^i_t(R).
    \end{equation}

    In conclusion, by summing all the estimates   above, we arrive at
    \begin{equation}\label{17}
        \frac{d}{dt}\mu_t(R,\zeta)\leq A_\varepsilon(R,\zeta)\tilde{m}^i_t(R)\leq A_\varepsilon(R,\zeta)\mu_t(R-\zeta,\zeta),
    \end{equation}
    where for $\varepsilon\in(0,\varepsilon_\sigma)$   with $\sigma\in(0,1)$ and $\varepsilon_\sigma$ as in Lemma \ref{thm-moment}, there exists $\tilde{C}>0$ such that  the coefficient can be chosen as
    \begin{equation}\label{18}
        A_\varepsilon(R,\zeta)=\frac{\tilde{C}}{\zeta}\left(1+\frac{R}{\rho^2}+\frac{1}{R^3|\ln\varepsilon|^\sigma}+\frac{1}{|\ln \varepsilon|}+\frac{1}{|\ln \varepsilon|^2}+\frac{1}{R^2|\ln \varepsilon|^\sigma}+\frac{1}{\zeta R^2|\ln \varepsilon|^\sigma}\right).
    \end{equation}
    Therefore, by  (\ref{17}), we get
    \begin{equation}\label{4-23}
        \mu_t(R,\zeta)\leq \mu_0(R,\zeta)+A_\varepsilon(R,\zeta)\int^t_0\mu_s(R-\zeta,  \zeta)ds\quad \forall t\in[0,T_\varepsilon].
    \end{equation}
    Let $\lfloor z\rfloor$ denote  the integer part of the positive number $z$ and take $n=\lfloor \ln |\ln \varepsilon\rfloor|$. Let $R_j=R-(j+1)\zeta$ and iterate the   inequality \eqref{4-23} $n$ times, from $R_0=R-\zeta$ to $R_n=R-(n+1)\zeta=R/2$. Since $\zeta=R/(2n+2)$ and $R_j\in [R/2,R]$, from the explicit expression  (\ref{18}), it is readily seen that if $\varepsilon$ is sufficiently small then there exists $C_A=C_A(|x_0|,a_i,\overline{d},h)>0$ such that if $\varepsilon$ small enough we have
    \begin{equation}
        A_\varepsilon(R_j,\zeta)\leq \frac{nC_A}{R}+\frac{nC_A}{\rho^2}.
    \end{equation}
    Therefore, for any $t\in[0,T_\varepsilon]$,
    \begin{equation}
        \mu_t(R-\zeta,\zeta)\leq \mu_0(R-\zeta,\zeta)+\sum^n_{j=1}\mu_0(R_j,\zeta)\frac{(A_\varepsilon(R,\zeta)t)^j}{j!}+\frac{A_\varepsilon(R,\zeta)^{n+1}}{n!}\int^t_0(t-s)^n\mu_s(R_{n+1},\zeta)ds.
    \end{equation}
    Since if $\varepsilon$ is sufficiently small $\Lambda_{i,\varepsilon}(0)\subset B(x_0+{P_i^0}/{|\ln \varepsilon|},\varepsilon)$,  then $\mu_0(R_j,\zeta)=0$ for any $j=0,\ldots,n$. Therefore
    \begin{equation*}
        \mu_t(R-\zeta,\zeta)\leq \frac{A_\varepsilon(R,\zeta)^{n+1}}{n!}\int^t_0(t-s)^n\mu_s(R_{n+1},\zeta)ds,
    \end{equation*}
    which combining
    \begin{equation*}
        \mu_s(R_{n+1},\zeta)\leq \tilde{m}^i_s(R_{n+1})\leq m^i_s(c_0'(R/2-\zeta))\leq \frac{|\ln \varepsilon|^2}{c_0'^2(R/2-\zeta)^2}J_{i,\varepsilon}(s)\leq\frac{9}{c_0'^2R^2|\ln \varepsilon|^{\sigma+1+b}},
    \end{equation*}
    leads to
    \begin{equation}
        \tilde{m}^i_t(R)\leq \mu_t(R-\zeta,\zeta)\leq \frac{9}{c_0'^2R^2|\ln \varepsilon|^{\sigma+1+b}}\frac{\left(n\cdot \frac{A_\varepsilon(R,\zeta)}{n}t\right)^{n+1}}{(n+1)!}\leq \frac{9}{c_0'^2R^2|\ln \varepsilon|^{\sigma+1+b}}\left(eC_A(\frac{1}{\rho^2}+\frac{1}{R})t\right)^{n+1}.
    \end{equation}
    Recalling that $m^i_t(R)\leq \tilde{m}^i_t(R/C_0')$, choosing $\tilde{T}_l=\frac{R\rho^2}{C_A(C_0'\rho^2+R)}e^{-(l+1)}\wedge T$, then we obtain  (\ref{19}).

\end{proof}

Having established preliminary  estimates in Proposition \ref{20}, we can further improve them as follows:
\begin{proposition}\label{23}
    Let $m^i_t(R)$ be as in Proposition \ref{20}. Then, given $R>0$, for each $l\in\mathbb{R}$ there is $\overline{T}_l\in (0,T]$ such that
  for any $\varepsilon\in(0,\varepsilon_0)$ sufficiently small,
    \begin{equation}\label{21}
        m^i_t(R)\leq\varepsilon^l\quad \forall t\in[0,\overline{T}_l\wedge T_\varepsilon].
    \end{equation}

\end{proposition}
\begin{proof}
    The strategy used in the proof of Proposition \ref{20} would give the stronger estimate  (\ref{21}) if we could choose $n=\lfloor|\ln \varepsilon|\rfloor$. But this means $h\sim |\ln\varepsilon|^{-1}$, which implies that the last term of  (\ref{18}) diverges as $\varepsilon$ vanishes. Notice that this term comes from \eqref{16}, and thus we need to improve the estimate of $A^1_2$.
    By Proposition \ref{20} and choosing $l=3-b$ we have
    \begin{equation*}
        m^i_t(c_0'R/2)\leq \frac{1}{|\ln \varepsilon|^4}\quad \forall t\in[0,\tilde{T}_2\wedge T_\varepsilon]
    \end{equation*}
    for any $\varepsilon$ small enough.Therefore,
    \begin{equation*}
        \begin{split}
            &\left|\iint_{\mathcal{A}_R^\complement\times \mathcal{A}_{R/2}^\complement}f(x,z,t)dzdx \right|+\left|\iint_{\mathcal{A}_{R/2}^\complement\times\mathcal{A}_{R/2}^\complement}f(x,z,t)dzdx\right|\\
            \leq&\overline{C}_K\left(\frac{|\ln \varepsilon|^2}{\zeta^2}+\frac{|\ln \varepsilon|}{\zeta}\right)\tilde{m}^i_t(R/2)\tilde{m}^i_t(R).
        \end{split}
    \end{equation*}
    Using this improved estimate, we finally have
    \begin{equation*}
        \frac{d}{dt}\mu_t(R,\zeta)\leq\tilde{A}_\varepsilon(R,\zeta)\mu_t(R-\zeta,\zeta),
    \end{equation*}
    with

    \begin{equation*}
        \tilde{A}_\varepsilon(R,\zeta)=\frac{\overline{C}}{\zeta}\left(1+\frac{R}{\rho^2}+\frac{1}{R^3|\ln\varepsilon|^\sigma}+\frac{1}{|\ln \varepsilon|}+\frac{1}{|\ln \varepsilon|^2}+\frac{1}{R^2|\ln \varepsilon|^\sigma}+\frac{1}{\zeta|\ln \varepsilon|^2}\right).
    \end{equation*}
    So there exists $\tilde{C}_A>0$ such that if $\varepsilon$ sufficiently small we have
    \begin{equation*}
        \tilde{A}_\varepsilon(R,h)\leq \frac{n\tilde{C}'_A}{R}+\frac{n\tilde{C}'_A}{\rho^2}\quad \tilde{C}'_A=\tilde{C}'_A(|x_0|,a_i,\overline{d},h).
    \end{equation*}
    Similarly, if we choose $\overline{T}_l=\frac{R\rho^2}{\tilde{C}_A(\rho^2+R)}e^{-(l+2)}\leq \frac{\rho^2R/C'_0}{\tilde{C}'_A(\rho^2+R/C'_0)}e^{-(l+2)}$, where $\tilde{C}_A=\tilde{C}_A(|x_0|,a_i,\overline{d},h)\geq C_A$, we will conclude  (\ref{21}).
\end{proof}

\section{Localization of vortices support and proof of the main theorem}
In the first part of this section we study the location and size of the support $\Lambda_{i,\varepsilon}(t)$. In order to obtain the necessary cancellation, we need to estimate the transformed distance $|DT(x_0)(x(t)-B^{i,\varepsilon}(t))|$, rather than $|x(t)-B^{i,\varepsilon}(t)|$ directly. In the second part we prove Theorem \ref{thm-main}.
\subsection{Localization of vortices support}
Recall that for any $z\in\mathbb{R}^2$,
\begin{equation}\label{5-1}
    c_0|z|\leq|DT(x_0)z|\leq C_0|z|.
\end{equation}
 Define
\begin{equation}
	R_t:=|\ln \varepsilon|\max\{|x-B^{i,\varepsilon}(t)|:x\in\Lambda_{i,\varepsilon}(t)\}.
\end{equation}
\begin{lemma}\label{25}
 Let $x(t)$ be a solution to $\dot{x}(t)=\frac{1}{|\ln\varepsilon|^{1-b}}v(x(t),t)$ with initial condition $x(0)\in\Lambda_{i,\varepsilon}(0)$ and suppose that at time $t\in(0,T_\varepsilon)$ it happens that
    \begin{equation}\label{time}
        |\ln \varepsilon||x(t)-B^{i,\varepsilon}(t)|=R_t.
    \end{equation}
Denote $\tilde{R}_t=|\ln \varepsilon||DT(x_0)(x(t)-B^{i,\varepsilon}(t))|$. Then, at that time $t$ such that \eqref{time} holds, there exists $\tilde{C}'=\tilde{C}'(|x_0|,a_i,\overline{d},h)$ such that for each fixed $\sigma\in(0,1)$ and any $\varepsilon$ small enough,
    \begin{equation}\label{22}
        \dot{\tilde{R}}_t\leq\tilde{C}'\left(1+\tilde{R}_t+\frac{1}{|\ln \varepsilon|^\sigma\tilde{R}_t^3}+\frac{|\ln\varepsilon|^b}{\varepsilon}\sqrt{m^i_t\left(\tilde{R}_t/2c_0\right)}\right).
    \end{equation}
\end{lemma}
\begin{proof}
For simplicity, denote $x_1=x(t)$. By the definitions and direct computations, we have
\begin{equation}\label{5-5}
\begin{split}
     &\quad \frac{d}{dt}|DT(x_0)(x-B^{i,\varepsilon}(t))|\\
     &=DT(x_0)\left(\frac{1}{|\ln\varepsilon|^{1-b}}v(x_1,t)-\dot{B}^{i,\varepsilon}(t)\right)\cdot\frac{DT(x_0)(x_1-B^{i,\varepsilon}(t))}{|DT(x_0)(x_1-B^{i,\varepsilon}(t))|}\\
     &=\frac{1}{|\ln\varepsilon|^{1-b}}DT(x_0)(v^i_K(x_1,t)+v^i_L(x_1,t)+v^i_S(x_1,t)+F^i(x_1,t))\cdot\frac{DT(x_0)(x_1-B^{i,\varepsilon}(t))}{|DT(x_0)(x_1-B^{i,\varepsilon}(t))|}\\
     &-\frac{1}{|\ln\varepsilon|^{1-b}\gamma_i}\int DT(x_0)(v^i_K(y,t)+v^i_L(y,t)+v^i_S(y,t)+F^i(y,t))\omega_{i,\varepsilon}(y,t)dy\cdot\frac{DT(x_0)(x_1-B^{i,\varepsilon}(t))}{|DT(x_0)(x_1-B^{i,\varepsilon}(t))|}.
\end{split}
\end{equation}

 For the term in \eqref{5-5} involving $v^i_K(x_1,t)$, using \eqref{decom v}, by direct computation we get
\begin{equation*}
    \begin{split}
        &DT(x_0)v^i_K(x_1,t)\cdot\frac{DT(x_0)(x_1-B^{i,\varepsilon}(t))}{|DT(x_0)(x_1-B^{i,\varepsilon}(t))|}\\
        =&\frac{DT(x_0)(x_1-B^{i,\varepsilon}(t))}{|DT(x_0)(x_1-B^{i,\varepsilon}(t)|)}\cdot\int H(x_1,y)DT(x_0)DT(x_1)^*\frac{(T(x_1)-T(y))^\perp}{|T(x_1)-T(y)|^2}\omega_{i,\varepsilon}(y,t)dy\\
        :=&H_1+H_2,
    \end{split}
\end{equation*}
where $H_1$ and $H_2$ are defined by the above integral restricted to the interior and exterior of the disk $B(B^{i,\varepsilon}(t),\frac{R_t}{2|\ln \varepsilon|})$, respectively. Similar to the proof of Proposition \ref{20}, by the initial condition,  and simple Taylor's expansion of $H(x,y)$ and $T(y)$ at $x_0$,  we obtain for $t\in(0,T_\varepsilon)$ that
\begin{equation*}
    \begin{split}
        H_1&=\frac{DT(x_0)(x_1-B^{i,\varepsilon}(t))}{|DT(x_0)(x_1-B^{i,\varepsilon}(t))|}\\
        &\qquad \qquad \cdot\int_{B(B^{i,\varepsilon}(t),\frac{R_t}{2|\ln \varepsilon|})} H(x_1,y)DT(x_0)DT(x_1)^*\frac{(T(x_1)-T(y))^\perp}{|T(x_1)-T(y)|^2}\omega_{i,\varepsilon}(y,t)dy\\
        &=\frac{DT(x_0)(x_1-B^{i,\varepsilon}(t))}{|DT(x_0)(x_1-B^{i,\varepsilon}(t))|}\\
        &\qquad \qquad\cdot\int_{B(B^{i,\varepsilon}(t),\frac{R_t}{2|\ln \varepsilon|})} H(x_0,x_0)DT(x_0)DT(x_0)^*\frac{(DT(x_0)(x_1-y))^\perp}{|DT(x_0)(x_1-y)|^2}\omega_{i,\varepsilon}(y,t)dy\\&+O\left(\frac{1}{|\ln \varepsilon|^{1+b}}\right).
    \end{split}
\end{equation*} And the above identity can be rewritten as follows:
    \begin{equation}
        H_1=C_{K,1}\int_{|y'|\leq\frac{R_t}{2|\ln \varepsilon|}}\frac{-\tilde x_1'\cdot \tilde y'^\perp}{|\tilde x_1'||\tilde x_1'-\tilde y' |^2}\tilde{\omega}_{i,\varepsilon}(y',t)dy'+O\left(\frac{1}{|\ln \varepsilon|^{1+b}}\right).
    \end{equation}
    Using the identity $\int \tilde y' \tilde{\omega}_{i,\varepsilon}(y',t)dy'=0$ due to the definition of $B^{i,\varepsilon}(t)$, we further decompose $H_1$ as follows:
    \begin{equation}
        H_1=H_1'-H_1''+O\left(\frac{1}{|\ln \varepsilon|^{1+b}}\right),
    \end{equation}
    where
    \begin{equation*}
        H_1'=-C_{K,1}\int_{|y'|\leq\frac{R_t}{2|\ln \varepsilon|}}\frac{(\tilde x_1'\cdot \tilde y'^\perp) \tilde y'\cdot(2\tilde  x_1'-\tilde y')}{|\tilde x_1'||\tilde x_1'-\tilde y' |^2|\tilde x_1'|^2}\tilde{\omega}_{i,\varepsilon}(y',t)dy',
    \end{equation*}
    \begin{equation*}
        H''_1=-C_{K,1}\int_{|y'|\geq\frac{R_t}{2|\ln \varepsilon|}}\frac{\tilde x_1'\cdot \tilde y'^\perp}{|\tilde x_1'|^3}\tilde{\omega}_{i,\varepsilon}(y',t)dy'.
    \end{equation*}
    We have  $|\tilde x_1'|\geq\frac{c_0R_t}{|\ln \varepsilon|}$, and  for $|y'|\leq\frac{R_t}{2|\ln \varepsilon|}$, there hold  $|2\tilde x_1'-\tilde y'|\leq |\tilde x_1'-\tilde y'|+|\tilde x_1'|$  and $|\tilde x_1'-\tilde y'|\geq\frac{c_0R_t}{2|\ln \varepsilon|}$. Therefore, we obtain
    \begin{equation}
        \begin{split}
            |H'_1|\leq  \overline{C}'_{K,1}\frac{|\ln \varepsilon|^3}{R_t^3}\int_{|y'|\leq\frac{R_t}{2|\ln \varepsilon|}}|\tilde y'|^2\tilde{\omega}_{i,\varepsilon}(y',t)dy' \leq \overline{C}'_{K,1}\frac{|\ln \varepsilon|^3}{R_t^3}J_{i,\varepsilon}(t).
        \end{split}
    \end{equation}
    And it follows easily from   Chebyshev's inequality that
    \begin{equation}
        \begin{split}
            |H''_1| \leq\tilde{C}'_H\frac{1}{|\tilde x_1'|^2}\int_{|y'|\geq\frac{R_t}{2|\ln \varepsilon|}}|\tilde y'|\tilde{\omega}_{i,\varepsilon}(y',t)dy' \leq\tilde{C}'_{H,1}\frac{|\ln \varepsilon|^3}{R_t^3}J_{i,\varepsilon}(t).
        \end{split}
    \end{equation}
    Then using  (\ref{15}) again, we have
    \begin{equation}
        |H_1|\leq \overline{C}'''_{K,1}\left(\frac{1}{R_t^3|\ln \varepsilon|^{\sigma+b}}+\frac{1}{|\ln \varepsilon|^{1+b}}\right).
    \end{equation}

    We   denote $\tilde{r}_i$ such that $\frac{M}{\varepsilon^2}\pi \tilde{r}_i^2=m^i_t(R_t/2)$.
    Then by the rearrangement inequality, we calculate
    \begin{equation}
        \begin{split}
         |H_2|&\leq    \left|\frac{DT(x_0)(x_1-B^{i,\varepsilon}(t))}{|DT(x_0)(x_1-B^{i,\varepsilon}(t))|}\right.\\
         &\left.\qquad \qquad\qquad\cdot\int_{B(B^{i,\varepsilon}(t),\frac{R_t}{2|\ln \varepsilon|})^\complement} H(x_1,y)DT(x_0)DT(x_1)^*\frac{(T(x_1)-T(y))^\perp}{|T(x_1)-T(y)|^2}\omega_{i,\varepsilon}(y,t)dy\right|\\
         &\leq C_{H,D}\int_{B(B^{i,\varepsilon}(t),\frac{R_t}{2|\ln \varepsilon|})^\complement}\frac{|\omega_{i,\varepsilon}(y,t)|}{|T(x_1)-T(y)|}dy\\
         &\leq \frac{C_{H,D}}{C_T}\int_{B(B^{i,\varepsilon}(t),\frac{R_t}{2|\ln \varepsilon|})^\complement}\frac{|\omega_{i,\varepsilon}(y,t)|}{|x_1-y|}dy\\
         &\leq\frac{C_{H,D}}{C_T}\frac{2\pi M}{\varepsilon^2}\int_0^{\tilde{r}_i}\frac{r}{r}dr
         =\frac{2C_{H,D}}{\varepsilon C_T}\sqrt{\pi Mm^i_t(R_t/2)}\\
         &\leq\frac{2C_{H,D}}{\varepsilon C_T}\sqrt{\pi Mm^i_t(\tilde{R}_t/2c_0)}.
        \end{split}
    \end{equation}

    As for the term in \eqref{5-5} containing $v^i_S(x_1,t)$, by using \eqref{estimate vs}, we find
    \begin{equation}
        \left|DT(x_0)v^i_S(x_1,t)\cdot\frac{DT(x_0)(x_1-B^{i,\varepsilon}(t))}{|DT(x_0)(x_1-B^{i,\varepsilon}(t))|}\right|\leq\frac{C_S|a_i|}{|\ln \varepsilon|^{1+b}}.
    \end{equation}
    Notice that there exists $C_H=C_H(|x_0|,h,\overline{d})$ such that $|DT(x_0)H(x,y)\frac{x^\perp}{2|X_1|^2}|\leq C_H$. Therefore, using the definition of $v^i_L$ in \eqref{decom v} and the estimate in \eqref{4-18}, we get
    \begin{equation}
        \begin{split}
            \left|DT(x_0)v^i_L(x_1,t)\cdot\frac{DT(x_0)(x_1-B^{i,\varepsilon}(t))}{|DT(x_0)(x_1-B^{i,\varepsilon}(t))|}\right|
            &\leq  C_H\left|\int \ln |T(x_1)-T(y)\omega_{i,\varepsilon}(y,t)dy|\right|\\
            &\leq \frac{C'_L}{|\ln \varepsilon|^b}.
        \end{split}
    \end{equation}

    On the other hand, by Taylor's expansion, it is easy to see that there exists $C'_H=C'_H(|x_0|,h,\overline{d})$ satisfying
    \begin{equation*}
        \left|DT(x_0)H(y,z)(DT(y)^*-DT(z)^*)\cdot\frac{(T(y)-T(z))^\perp}{|T(y)-T(z)|^2}\right|\leq C'_{H,D}|y-z|\cdot\frac{|y-z|}{|y-z|^2}=C'_{H,D}.
    \end{equation*} Then, by symmetry properties and similar argument as above, we have
    \begin{equation}
        \begin{split}
            &\left|\frac{1}{\gamma_i}\int DT(x_0)(v^i_K(y,t)+v^i_L(y,t)+v^i_S(y,t))\cdot\frac{DT(x_0)(x_1-B^{i,\varepsilon}(t))}{|DT(x_0)(x_1-B^{i,\varepsilon}(t))|}\omega_{i,\varepsilon}(y,t)dy \right|\\
            \leq&\left|\frac{1}{2\gamma_i}\iint DT(x_0)H(y,z)(DT(y)^*-DT(z)^*)\cdot\frac{(T(y)-T(z))^\perp}{|T(y)-T(z)|^2} \omega_{i,\varepsilon}(y,t)\omega_{i,\varepsilon}(z,t)dydz\right|\\
            +&\left|\frac{1}{\gamma_i}\iint DT(x_0)\frac{y^\perp}{2|Y|^2}H(y,z)\ln |T(y)-T(z)|\omega_{i,\varepsilon}(y,t)\omega_{i,\varepsilon}(z,t)dydz\right|+\frac{C_S|a_i|}{|\ln \varepsilon|^{1+b}}\\
            \leq& \frac{|a_i|C'_{H,D}}{|\ln \varepsilon|^{1+b}}+\frac{C'_L}{|\ln \varepsilon|^b}+\frac{C_S|a_i|}{|\ln \varepsilon|^{1+b}}.
        \end{split}
    \end{equation}
    Finally, for the term involving $F^i$ in \eqref{5-5}, by  (\ref{14}), we compute
    \begin{equation}
        \begin{split}
            &\left|DT(x_0)F^i(x,t)\cdot\frac{DT(x_0)(x_1-B^{i,\varepsilon}(t))}{|DT(x_0)(x_1-B^{i,\varepsilon}(t))|}\right.\\
            &\qquad\qquad \left.-\frac{1}{\gamma_i}\int DT(x_0)F^i(x,t)\omega_{i,\varepsilon}(y,t)dy\cdot\frac{DT(x_0)(x_1-B^{i,\varepsilon}(t))}{|DT(x_0)(x_1-B^{i,\varepsilon}(t))|}\right|\\
            \leq&\frac{C_0}{\gamma_i}\int |F^i(x,t)-F^i(y,t)|\omega_{i,\varepsilon}(y,t)|dy\\
            \leq&\frac{C'_Fa}{2\rho^2|\ln \varepsilon|}R_t\leq\frac{C'_Fa}{2\rho^2|\ln \varepsilon|C_0}\tilde{R}_t.
        \end{split}
    \end{equation}
    In conclusion, adding the above estimates together we get the desired inequality (\ref{22}) and hence finish the proof.
\end{proof}

\begin{proposition}\label{27}
    There exists $T'_{\rho}\in(0,T]$ such that, for any $\varepsilon$ small enough,
    \begin{equation}\label{24}
        \Lambda_{i,\varepsilon}(t)\subset B(B^{i,\varepsilon}(t),\frac{\rho}{2|\ln \varepsilon|})\quad \forall t
        \in[0,T'_\rho\wedge T_\varepsilon].
    \end{equation}
\end{proposition}
\begin{proof}
    Let $\overline{T}_3$ be the constant as in Proposition \ref{23} with the choice $R=\frac{c_0\rho}{10C_0}$ and $l=3$. Here $c_0$ stands for the constant in \eqref{5-1}.
    Define
    $$t_1:=\sup\{t\in[0,\overline{T}_3\wedge T_\varepsilon]:\tilde{R}_s\leq\frac{c_0\rho}{2},\forall s\in[0,t]\},$$
    where $\tilde{R}_s$ is given in Lemma \ref{25}.

    Notice that when $\tilde{R}_s\leq\frac{c_0\rho}{2}$, we have $R_s\leq \frac{1}{c_0}\tilde{R}_s\leq\frac{\rho}{2}$.

    If $t_1=\overline{T}_3\wedge T_\varepsilon$,  then (\ref{24}) is proved with $T'_\rho=\overline{T}_3$. Otherwise, if $t_1<\overline{T}_3\wedge T_\varepsilon$, we define
    \begin{equation*}
        t_0:=\inf\left\{t\in[0,t_1]:\tilde{R}_s>\frac{c_0\rho}{5},\forall s\in[t,t_1]\right\}.
    \end{equation*}
    Since $\tilde{R}_0\leq 2C_0\varepsilon$, we have $t_0>0$ provided $\ep$ sufficiently small. Moreover, by the definition, it is obvious that $\tilde{R}_{t_1}=\frac{c_0\rho}{2}$, $\tilde{R}_{t_0}=\frac{c_0\rho}{5}$, and $\tilde{R}_{t}\in[\frac{c_0\rho}{5},\frac{c_0\rho}{2}]$ for any $t\in[t_0,t_1]$. In particular, one has $R_{t}\in[\frac{c_0\rho}{5C_0},\frac{\rho}{2}]$, and by Proposition \ref{23},
    \begin{equation*}
        m^i_t\left(\tilde{R}_t/2c_0\right)\leq m^i_t\left(c_0\rho/10C_0\right)\leq \varepsilon^3\quad \forall t\in[t_0,t_1].
    \end{equation*}

    Let $R(t)$ satisfy
    \begin{equation}\label{26}
        \dot{R}(t)=\tilde{C}'\left(2+\rho^{-2}R(t)+\frac{1}{|\ln \varepsilon|^\sigma R(t)^3}+g_\varepsilon(t)\right),\quad R(t_0)=\frac{c_0\rho}{4},
    \end{equation} where $g_\varepsilon(t)$ is any smooth function which is an upper bound for the last term in equation (\ref{22}).
    We claim that if $\varepsilon$ is small enough then
    \begin{equation}\label{claim}
    	  \Lambda_{i,\varepsilon}(t)\subset B(B^{i,\varepsilon}(t),c_0^{-1}R(t)), \ \ \ \text{ for any}\ \   t\in(t_0,t_1).
    \end{equation}
    Indeed, this is true for $t=t_0$ because $R_{t_0}\leq\frac{\tilde{R}_{t_0}}{c_0}=\frac{\rho}{5}<\frac{\rho}{4}=\frac{R(t_0)}{c_0}$. Suppose that claim \eqref{claim} is not true and  there exists a first time $t_*\in(t_0,t_1)$ such that $\tilde{R}_{t_*}=R(t_*)$, then we have $\dot{\tilde{R}}_{t_*}\geq \dot{R}(t_*)$. But according to  (\ref{22}) and (\ref{26}), we have $\dot{\tilde{R}}_{t_*}> \dot{R}(t_*)$, in contradiction of $t_*$ as the first time at which the graph of $t\mapsto\tilde{R}_t$ crosses the one of $t\mapsto R(t)$. This proves the claim \eqref{claim}.

    When $\varepsilon$ is small enough we have $g_\varepsilon(t)\leq 2\tilde{C}'\varepsilon^{\frac{1}{2}}$ and thus
    \begin{equation*}
        \dot{R}(t)\leq 3\tilde{C}'+\tilde{C}'\rho^{-2}R(t)\quad \forall t\in[t_0,t_1].
    \end{equation*}
    Notice $\rho\leq \overline{d}$, by adjust the dependency relationship of $\tilde{C}'$ on $\overline{d}$, we may assume that
    \begin{equation*}
        \dot{R}(t)\leq 3\rho^{-2}\tilde{C}'+\rho^{-2}\tilde{C}'R(t)\quad \forall t\in[t_0,t_1].
    \end{equation*}
    Then we get
    \begin{equation*}
        R(t_1)\leq e^{\tilde{C}'\rho^{-2}(t_1-t_0)}R(t_0)+3\left(e^{\tilde{C}'\rho^{-2}(t_1-t_0)}-1\right),
    \end{equation*}
    which implies
    \begin{equation}
        t_1-t_0\geq \frac{\rho^2}{\tilde{C}'}\ln \frac{\tilde{C}'R(t_1)+3\tilde{C}'\rho^2}{\tilde{C}'R(t_0)+3\tilde{C}'\rho^2}.
    \end{equation}
    Therefore, as $R(t_1)\geq c_0\rho/2$ and $R(t_0)= c_0\rho/4$, the claim follows with
    \begin{equation}\label{32}
        T_\rho'=\frac{\rho^2}{\tilde{C}'}\ln \frac{\tilde{C}'c_0\rho/2+3\tilde{C}'\rho^2}{\tilde{C}'c_0\rho/4+3\tilde{C}'\rho^2}\wedge\overline{T}_3=\frac{\rho^2}{a'}\ln \frac{2b+C\rho}{b+C\rho}\wedge\overline{T}_3,
    \end{equation}
    with positive constants $a'=a'(|x_0|,h,a_i)$ , $b=b(|x_0|,h,a_i)$ and
    \begin{equation}\label{33}
        \overline{T}_3=\frac{\rho^2R}{\tilde{C}_A(R+\rho^2)}e^{-5}\wedge T.
    \end{equation}
    The proposition is thus proved.
\end{proof}

\subsection{Conclusion of the proof of Theorem \ref{thm-main}}
In this subsection we finish the proof of Theorem \ref{thm-main}. Let  $q^{i,\varepsilon}(t)$ be as in Proposition \ref{11} and $P_i(t)$ be as in Theorem \ref{thm-main}. Denote  $\overline{P}_i(t)=x_0+\frac{P_i(t)}{|\ln \varepsilon|}$. To prove Theorem \ref{thm-main} we only need to show  that the following conclusions are true
\begin{equation} \label{29}  \lim_{\varepsilon\xrightarrow{}0}\max_{i\in\{1,\ldots,N\}}\max_{t\in[0,T_\varepsilon]}|\ln \varepsilon||B^{i,\varepsilon}(t)-\overline{P}_i(t)|=0,
\end{equation}
\begin{equation}\label{28}   \lim_{\varepsilon\xrightarrow{}0}\max_{i\in\{1,\ldots,N\}}\max_{t\in[0,T_\varepsilon]}|\ln \varepsilon||B^{i,\varepsilon}(t)-q^{i,\varepsilon}(t)|=0.
\end{equation}
 Indeed, if we choose $R=R_\varepsilon=e^{\sqrt{|\ln \varepsilon|}}$ in Proposition \ref{11} and combine the result of Proposition \ref{27}, according to the definition of $T'_\rho$, $T_\varepsilon$, the statement of Proposition \ref{11} is proved with $T_\rho=T'_\rho$, $P^\varepsilon_i(t)=q^{i,\varepsilon}(t)$, and $\rho_\varepsilon=\varepsilon R_\varepsilon$.
\begin{proof}
    We prove  (\ref{28}) first. Choosing $R=R_\varepsilon=e^{\sqrt{|\ln \varepsilon|}}$ in Proposition \ref{11}, we use  (\ref{10}) to compute
    \begin{equation*}
        \begin{split}
            |B^{i,\varepsilon}(t)-q^{i,\varepsilon}(t)|&\leq\frac{1}{\gamma_i}\int|x-q^{i,\varepsilon}(t)|\omega_{i,\varepsilon}(x,t)dx\\
            &\leq\frac{\varepsilon R_\varepsilon}{|\ln \varepsilon|}+\frac{1}{\gamma_i}\int_{B(q^{i,\varepsilon}(t),\frac{\varepsilon R_\varepsilon}{|\ln \varepsilon|})^\complement}|x-q^{i,\varepsilon}(t)|\omega_{i,\varepsilon}(x,t)dx\\
            &\leq\frac{\varepsilon R_\varepsilon}{|\ln \varepsilon|}+\frac{\overline{C}_4\ln |\ln \varepsilon|}{|\gamma_i||\ln \varepsilon|^{1+b}\sqrt{|\ln \varepsilon|}}|B^{i,\varepsilon}(t)-q^{i,\varepsilon}(t)|\\
            &+\frac{1}{\gamma_i}\int_{B(q^{i,\varepsilon}(t),\frac{\varepsilon R_\varepsilon}{|\ln \varepsilon|})^\complement}|x-B^{i,\varepsilon}(t)|\omega_{i,\varepsilon}(x,t)dx.
        \end{split}
    \end{equation*}
    Taking $\varepsilon$ sufficiently small such  that $\frac{\overline{C}_4\ln |\ln \varepsilon|}{|a_i|\sqrt{|\ln \varepsilon|}}\leq \frac{1}{2}$ and using H\"older's inequality, we then obtain
    \begin{equation*}
        |B^{i,\varepsilon}(t)-q^{i,\varepsilon}(t)|\leq\frac{2\varepsilon R_\varepsilon}{|\ln \varepsilon|}+2\sqrt{\frac{J_{i,\varepsilon}(t)}{|\gamma_i|}},
    \end{equation*}
    which, combined with (\ref{15}), gives
    \begin{equation}
        |\ln \varepsilon||B^{i,\varepsilon}(t)-q^{i,\varepsilon}(t)|\leq2\varepsilon e^{\sqrt{|\ln \varepsilon|}}+\frac{1}{|\ln \varepsilon|^{\frac{\sigma}{2}}}.
    \end{equation}
    This proves (\ref{28}).

    \medskip

    Next we prove  (\ref{29}). Denote $\tilde{z}=DT(x_0)z$ for any $z\in\mathbb{R}^2$. Clearly,  $|\tilde{z}|\leq C_0|z|$. Let
    \begin{equation*}
        \begin{split}
            &\Delta(t):=\sum_i|B^{i,\varepsilon}(t)-\overline{P}_i(t)|^2,\\
            &\tilde{\Delta}(t):=\sum_i|\tilde{B}^{i,\varepsilon}(t)-\tilde{\overline{P}}_i(t)|^2,\quad \forall t\in[0,T_\varepsilon].
        \end{split}
    \end{equation*}
    Recall that the point $\tilde P_i$ solves the following ODE:
    \begin{equation*}
        \dot{ \tilde{P}}_i=A\sum_{j\neq i}a_j\frac{(\tilde{P}_i-\tilde{P}_j)^\perp}{|\tilde{P}_i-\tilde{P}_j|^2}-a_iB\begin{pmatrix}
        0\\1
    \end{pmatrix},\quad i=1,...,N,
    \end{equation*}
    with constants $A$ and $B$ defined in \eqref{def ab}.
    By the definition of \eqref{center}, one has
    \begin{equation*}
        \dot{\tilde{B}}^{i,\varepsilon}(t)=\frac{DT(x_0)}{|\ln\varepsilon|^{1-b}\gamma_i}\int(v^i_K(x,t)+v^i_L(x,t)+v^i_S(x,t)+F^i(x,t))\omega_{i,\varepsilon}(x,t)dx.
    \end{equation*}
    It is clear that $\Delta(t)\leq C_0^2\tilde{\Delta}(t)$ and
    \begin{equation*}
    \begin{split}
        \dot{\tilde{\Delta}}(t)&=2\sum_i(\tilde{B}^{i,\varepsilon}(t)-\tilde{\overline{P}}_i(t))\cdot(\dot{\tilde{B}}^{i,\varepsilon}(t)-\frac{\dot{\tilde{P}}_i(t)}{|\ln \varepsilon|})\\
        &=2\sum^5_{\ell=1}\sum_i(\tilde{B}^{i,\varepsilon}(t)-\tilde{\overline{P}}_i(t))\cdot D^i_\ell(t),
    \end{split}
    \end{equation*}
    where
    \begin{align*}
            D^i_1=&\frac{DT(x_0)}{|\ln\varepsilon|^{1-b}\gamma_i}\sum_{j\neq i}\int v^j_R(R,t)\omega_{i,\varepsilon}(x,t)dx,\\
            D^i_2=&\frac{DT(x_0)}{|\ln\varepsilon|^{1-b}\gamma_i}\sum_{j\neq i}\int v^j_L(R,t)\omega_{i,\varepsilon}(x,t)dx\\
            =&\frac{DT(x_0)}{|\ln\varepsilon|^{1-b}\gamma_i}\sum_{j\neq i}\iint\frac{x^\perp}{2|X|^2}H(x,y)\ln |T(x)-T(y)|\omega_{j,\varepsilon}(y,t)\omega_{i,\varepsilon}(x,t)dydx,\\
            D^i_3=&\frac{1}{|\ln\varepsilon|^{1-b}\gamma_i}\sum_{j\neq i}\iint\left(DT(x_0)H(x,y)DT(x)^*\frac{(T(x)-T(y))^\perp}{|T(x)-T(y)|^2}-\frac{A}{|\ln \varepsilon|}\frac{|\ln \varepsilon|^2(\tilde{P}_i-\tilde{P}_j)^\perp}{|\tilde{P}_i-\tilde{P}_j|^2}\right)\\
            &\omega_{j,\varepsilon}(y,t)\omega_{i,\varepsilon}(x,t)dydx,\\
            D^i_4=&\frac{DT(x_0)}{|\ln\varepsilon|^{1-b}\gamma_i}\iint H(x,y)DT(x)^*\frac{(T(x)-T(y))^\perp}{|T(x)-T(y)|^2}\omega_{i,\varepsilon}(y,t)\omega_{i,\varepsilon}(x,t)dydx,\\
            D^i_5=&\frac{DT(x_0)}{|\ln\varepsilon|^{1-b}\gamma_i}\int v^i_L(R,t)\omega_{i,\varepsilon}(x,t)dx+\frac{1}{|\ln \varepsilon|}\frac{a_i\tau(|x_0|^2)|x_0|}{4\pi h\sqrt{h^2+|x_0|^2}}\begin{pmatrix}
        0\\1
    \end{pmatrix}.
    \end{align*}
    By the estimates of $v^j_R$, there exists a positive $\overline{C}_R=\overline{C}_R(|x_0|,h,a_i)$ such that $|DT(x_0)\nabla^\perp S_{K} (x,y)|\leq \overline{C}_R$ for any $x,y\in \cup_i\Lambda_{i,\varepsilon}(t)$, then we have
    \begin{equation}
        \left|\sum_i(\tilde{B}^{i,\varepsilon}(t)-\tilde{\overline{P}}_i(t))\cdot D^i_1(t)\right|\leq \frac{\overline{C}_R}{|\ln \varepsilon|^2}\sqrt{\tilde{\Delta}(t)}.
    \end{equation}
    Notice that there exists a positive $\overline{C}_L=\overline{C}_L(|x_0|,h,a_i)$ such that $|H(x,y)\ln |T(x)-T(y)||\leq \overline{C}_L\ln |x-y|^{-1}$ for any $x\in\Lambda_{i,\varepsilon}(t)$ and $y\in\Lambda_{j,\varepsilon}(t)$, $i\neq j$, we have
    \begin{equation*}
        \begin{split}
           \left |\int H(x,y)\ln |T(x)-T(y)|\omega_{j,\varepsilon}(y,t)dy\right|&\leq\overline{C}_L\int \ln \frac{|\ln \varepsilon|}{2\rho}|\omega_{j,\varepsilon}(y,t)|dy\\
            &\leq\frac{\overline{C}_L|a_i|\ln |\ln \varepsilon|}{|\ln \varepsilon|^{1+b}}+\frac{\overline{C}_L|a_i||\ln 2\rho|}{|\ln \varepsilon|^{1+b}}.
        \end{split}
    \end{equation*}
    So we get
    \begin{equation}
        \left|\sum_i(\tilde{B}^{i,\varepsilon}(t)-\tilde{\overline{P}}_i(t))\cdot D^i_2(t)\right|\leq\frac{2\overline{C}_L|a_i|\ln |\ln \varepsilon|}{|\ln \varepsilon|^2}\sqrt{\tilde{\Delta}(t)}.
    \end{equation}
    As for $D^i_4$, similar to the proof of Lemma \ref{25}, by symmetry we have
    \begin{equation*}
        |D^i_4|\leq \frac{C'_{H,D}|a_i|}{|\ln \varepsilon|^2},
    \end{equation*}
    then
    \begin{equation}
        \left|\sum_i(\tilde{B}^{i,\varepsilon}(t)-\tilde{\overline{P}}_i(t))\cdot D^i_4(t)\right|\leq \frac{C'_{H,D}|a_i|}{|\ln \varepsilon|^2}\sqrt{\tilde{\Delta}(t)}.
    \end{equation}
    In order to estimate the term containing $D^i_5(t)$, without loss of generality, we assume that $a_i>0$ and denote $x=(x^1,x^2)$. If $x_0\neq0$, noticing that $\frac{|x^1|}{2|X|^2}\leq \frac{|x_0|}{|x_0|^2+h^2}$ for $\varepsilon$ small enough, by the initial condition and the definition of $T_\varepsilon$, we get the following upper bounds
    \begin{equation*}
        \begin{split}
            &-\frac{\tau(r_0^2)}{|\ln\varepsilon|^{1-b}\gamma_i}\iint\frac{x^1}{2|X|^2}H(x,y)\ln |T(x)-T(y)|\omega_{j,\varepsilon}(y,t)\omega_{i,\varepsilon}(x,t)dydx\\
            \leq&\frac{\tau(r_0^2)r_0}{|\ln\varepsilon|^{1-b}\gamma_i(r_0^2+h^2)}\iint H(x_0,x_0)\ln |C_T(x-y)|^{-1}\omega_{j,\varepsilon}(y,t)\omega_{i,\varepsilon}(x,t)dydx+\frac{\overline{C}_5}{|\ln \varepsilon|^2}\\
            \leq&\frac{a_i\tau(r_0^2)r_0}{4\pi h|\ln \varepsilon|\sqrt{r_0^2+h^2}}+\frac{\tilde{C}_5}{|\ln \varepsilon|^2},
        \end{split}
    \end{equation*}
    and\begin{equation*}
        \begin{split}
            &\left|\frac{1}{\gamma_i}\left(\tau(r_0^2)+\frac{r_0^2}{h^2+h\sqrt{h^2+r_0^2}}\right)\iint\frac{x^2}{2|X|^2}H(x,y)\ln |T(x)-T(y)|\omega_{j,\varepsilon}(y,t)\omega_{i,\varepsilon}(x,t)dydx\right|\\
            \leq&\frac{C'_5}{|\ln \varepsilon|^{1+b}}.
        \end{split}
    \end{equation*}
    On the other hand, by choosing $R_\varepsilon=\exp(\sqrt{|\ln \varepsilon|\ln |\ln \varepsilon|})$, we have
    \begin{equation*}
        \begin{split}
           &- \frac{\tau(r_0^2)}{\gamma_i}\iint\frac{x^1}{2|X|^2}H(x,y)\ln |T(x)-T(y)|\omega_{j,\varepsilon}(y,t)\omega_{i,\varepsilon}(x,t)dydx\\
           \geq&- \frac{\tau(r_0^2)}{\gamma_i}\ln \frac{2C'_T\varepsilon R_\varepsilon}{|\ln \varepsilon|^\frac{1+b}{2}}\iint_{B\left(q^{i,\varepsilon}(t),\frac{\varepsilon}{|\ln \varepsilon|}R_\varepsilon\right)^2}\frac{x^1}{2|X|^2}H(x,y)\omega_{j,\varepsilon}(y,t)\omega_{i,\varepsilon}(x,t)dydx\\
           \geq&-\frac{r_0\tau(r_0^2)\ln (\varepsilon R_\varepsilon)}{2(r_0^2+h^2)\gamma_i}\frac{\sqrt{r_0^2+h^2}}{2\pi h}\left(\int_{B\left(q^{i,\varepsilon}(t),\frac{\varepsilon}{|\ln \varepsilon|^{(1+b)/2}}R_\varepsilon\right)}\omega_{j,\varepsilon}(y,t)dy\right)^2-\frac{\ln |\ln \varepsilon|C^*_L}{|\ln \varepsilon|^{1+b}}.
        \end{split}
    \end{equation*}
    Taking $\varepsilon$ small enough such that $R_\varepsilon>\exp(C_3\ln |\ln \varepsilon|)$, then by applying Proposition \ref{11}, we get
    \begin{equation*}
        \int_{B\left(q^{i,\varepsilon}(t),\frac{\varepsilon}{|\ln \varepsilon|^{(1+b)/2}}R_\varepsilon\right)}\omega_{j,\varepsilon}(y,t)dy\geq \gamma_i-\frac{C_4}{|\ln \varepsilon|^{1+b}}\sqrt{\frac{\ln |\ln \varepsilon|}{|\ln \varepsilon|}},\quad \forall t\in[0,T_\varepsilon].
    \end{equation*}
     Since $\left|\frac{1}{|\ln \varepsilon|}\ln \frac{1}{\varepsilon R_\varepsilon}-1\right|\leq \frac{\ln R_\varepsilon}{|\ln \varepsilon|}=\sqrt{\frac{\ln |\ln \varepsilon|}{|\ln \varepsilon|}}$, we obtain
    \begin{equation*}
        \begin{split}
            &- \frac{\tau(r_0^2)}{|\ln\varepsilon|^{1-b}\gamma_i}\iint\frac{x^1}{2|X|^2}H(x,y)\ln |T(x)-T(y)|\omega_{j,\varepsilon}(y,t)\omega_{i,\varepsilon}(x,t)dydx\\
            \geq&\frac{a_i\tau(r_0^2)r_0}{4\pi h|\ln \varepsilon|\sqrt{r_0^2+h^2}}-\frac{C^*_L\sqrt{|\ln \varepsilon|\ln |\ln \varepsilon|}}{|\ln \varepsilon|^2}.
        \end{split}
    \end{equation*}
    If $x_0=0$, then $B=0$, and $D^i_5=\frac{I_2}{\gamma_i}\int v^i_L(R,t)\omega_{i,\varepsilon}(x,t)dx$, by (\ref{3-4}) we have $\frac{|x^1|}{2|X|^2}\leq \frac{\overline{d}+\rho}{2h^2|ln\varepsilon|}$, and
    \begin{equation*}
        \left|\frac{1}{\gamma_i}\iint\frac{x^1}{2|X|^2}H(x,y)\ln |T(x)-T(y)|\omega_{i,\varepsilon}(y,t)\omega_{i,\varepsilon}(x,t)dydx\right|\leq\frac{\tilde{C}_5'}{|ln\varepsilon|^{1+b}}.
    \end{equation*}
    Combining the upper bound estimate, we get
    \begin{equation*}
        |D^i_5(t)|\leq\frac{C_{5,L}\sqrt{|\ln \varepsilon|\ln |\ln \varepsilon|}}{|\ln \varepsilon|^2},
    \end{equation*}
    which implies that
    \begin{equation}
        \left|\sum_i(\tilde{B}^{i,\varepsilon}(t)-\tilde{\overline{P}}_i(t))\cdot D^i_5(t)\right|\leq \frac{C_{5,L}\sqrt{|\ln \varepsilon|\ln |\ln \varepsilon|}}{|\ln \varepsilon|^2}\sqrt{\tilde{\Delta}(t)}.
    \end{equation}

    For the term involving $D^i_3(t)$, direct computations give
    \begin{align*}
            |D^i_3(t)|\leq&\frac{1}{|\ln\varepsilon|^{1-b}\gamma_i}\sum_{j\neq i}\iint\left|DT(x_0)H(x_0,x_0)DT(x_0)^*\frac{(T(x)-T(y))^\perp}{|T(x)-T(y)|^2}-A\frac{|\ln \varepsilon|(\tilde{P}_i-\tilde{P}_j)^\perp}{|\tilde{P}_i-\tilde{P}_j|^2}\right|\\
             &\times|\omega_{j,\varepsilon}(y,t)||\omega_{i,\varepsilon}(x,t)|dydx+\frac{C^*_3}{|\ln \varepsilon|^2}\\
            \leq&\left|\frac{1}{|\ln\varepsilon|^{1-b}\gamma_i}\sum_{j\neq i}\iint A\left(\frac{(\tilde{x}-\tilde{y})^\perp}{|\tilde{x}-\tilde{y}|^2}-\frac{(\tilde{\overline{P}}_i-\tilde{\overline{P}}_j)^\perp}{|\tilde{\overline{P}}_i-\tilde{\overline{P}}_j|^2}\right)\omega_{j,\varepsilon}(y,t)\omega_{i,\varepsilon}(x,t)dydx\right|+\frac{C^*_3}{|\ln\varepsilon|^2}\\
            \leq&\left|\frac{|\ln \varepsilon|^{2+2b}}{a_i\rho^2}\sum_{j\neq i}\iint A\left(|\tilde{x}-\tilde{\overline{P}}_i|+|\tilde{y}-\tilde{\overline{P}}_j|\right)\omega_{j,\varepsilon}(y,t)\omega_{i,\varepsilon}(x,t)dydx\right|+\frac{C^*_3}{|\ln\varepsilon|^2}\\
            \leq& \frac{|\ln \varepsilon|^{1+b}}{\rho^2}A\sum_{j\neq i}|\gamma_j|\left(|\tilde{B}^{i,\varepsilon}(t)-\tilde{\overline{P}}_i|+|\tilde{B}^{j,\varepsilon}(t)-\tilde{\overline{P}}_j|+C_{D,T}\sqrt{\frac{J_{i,\varepsilon}(t)}{|\gamma_i|}}+C_{D,T}\sqrt{\frac{J_{j,\varepsilon}(t)}{|\gamma_j|}}\right)\\
            &+\frac{C^*_3}{|\ln\varepsilon|^2}.
    \end{align*}
    Thus using (\ref{15}), we get
    \begin{equation}
        \left|\sum_i(\tilde{B}^{i,\varepsilon}(t)-\tilde{\overline{P}}_i(t))\cdot D^i_3(t)\right|\leq\frac{\tilde{C}_3a}{\rho^2}\left(\tilde{\Delta}(t)+\frac{1}{|\ln \varepsilon|^{1+\frac{\sigma}{2}}}\sqrt{\tilde{\Delta}(t)}\right).
    \end{equation}
    Finally, letting $\overline{\Delta}(t):=|\ln \varepsilon|^2\tilde{\Delta}(t)$, then combining all the estimates above, we can find positive $C_\alpha=C_\alpha(|x_0|,h,a_i,\rho,\overline{d})$ such that
    \begin{equation}
        \dot{\overline{\Delta}}(t)\leq C_\alpha\overline{\Delta}(t)+C_\alpha\left(\frac{1}{|\ln \varepsilon|^\frac{\sigma}{2}}+\frac{1}{|\ln \varepsilon|}+\frac{\sqrt{|\ln \varepsilon| \ln |\ln \varepsilon|}}{|\ln \varepsilon|}\right)\sqrt{\overline{\Delta}(t)},\quad \forall t\in[0,T_\varepsilon].
    \end{equation}
    According to the initial condition we have $\overline{\Delta}(0)\leq4NC_{D,T}\varepsilon^2$. Then (\ref{29}) follows immediately from the above differential inequality.

\end{proof}

\section{An example of helical vortices' leapfrogging phenomenon}\label{leapfrogging}

Now let us consider the case where there are two helices only. We describe the dynamical system  \eqref{ds} for $N=2$ and suppose that $\gamma_1+\gamma_2=\frac{a_1+a_2}{|\ln \varepsilon|^2}\neq 0$. Adopting  new variables,
\begin{equation*}
    x=\tilde{P}_1-\tilde{P}_2,\quad y=\frac{a_1\tilde{P}_1+a_2\tilde{P}_2}{a_1+a_2},
\end{equation*}
we then derive an equivalent system
\begin{equation}\label{31}
    \left\{
\begin{array}{ll}
\dot{x}=A(a_1+a_2)\nabla^\perp \ln |x|-\frac{(a_1-a_2)\tau(r_0^2)r_0}{4\pi h\sqrt{r_0^2+h^2}}\begin{pmatrix}
        0\\1
    \end{pmatrix},\\
    \dot{y}=-\frac{(a_1^2+a_2^2)\tau(r_0^2)r_0}{4\pi h(a_1+a_2)\sqrt{r_0^2+h^2}}\begin{pmatrix}
        0\\1
    \end{pmatrix}.
\end{array}
\right.
\end{equation}
We furthermore assume that $a_1>|a_2|$, and we notice that $x$ is governed by the canonical system $\dot{x}=\nabla^\perp\mathcal{H}(x)$ of Hamiltonian
\begin{equation}
\begin{split}
    \mathcal{H}(x)&=(a_1+a_2)A\ln |x|-\frac{(a_1-a_2)\tau(r_0^2)r_0}{4\pi h\sqrt{r_0^2+h^2}}x_1\\
    &=\frac{a_1+a_2}{4\pi}A_1\ln |x|^2-\frac{a_1-a_2}{4\pi}B_1x_1,\quad x=(x_1,x_2),
\end{split}
\end{equation}
with constants $A_1=\frac{\tau^2(r_0^2)\left(h^2+r_0^2+\sqrt{h^2+r_0^2}\right)\sqrt{h^2+r_0^2}}{h\left(h^2+h\sqrt{h^2+r_0^2}\right)}$ and $B_1=\frac{\tau(r_0^2)r_0}{h\sqrt{h^2+r_0^2}}$.
If $x_0\neq0$, this Hamiltonian system has a unique equilibrium, corresponding to the critical point $x^*=(2a',0)$, where
\begin{equation*}
    a'=\frac{(a_1+a_2)A_1}{(a_1-a_2)B_1}.
\end{equation*}
The phase portrait of this Hamiltonian system can be described by the energy level sets $\{x:\mathcal{H}(x)=E\}$. Then we recast the equation $\mathcal{H}(x)=E$
\begin{equation}\label{36}
    \begin{split}
        x_2=\pm f(x_1),\quad \text{with}\quad f(x_1)=\sqrt{C_E\exp\left(\frac{x_1}{a'}\right)-x^2_1}
    \end{split}
\end{equation}
where
\begin{equation*}
    C_E=\exp\left(\frac{4\pi E}{(a_1+a_2)A_1}\right).
\end{equation*}
Let
\begin{equation*}
    C^*=\exp\left(\frac{4\pi \mathcal{H}(x^*)}{(a_1+a_2)A_1}\right)=\left(\frac{2a'}{e}\right)^2.
\end{equation*}
By direct calculation we can see that
\begin{equation}
        \text{Dom}(f):=\{x_1:|x_1|\leq\sqrt{C_E}\exp(x_1/2a')\}=\left\{
\begin{array}{ll}
\left[\eta_1,\eta_2\right]\cup\left[\eta_3,+\infty\right) &if\quad C_E<C^*,\\
\left[\overline{\eta},+\infty\right)\quad &if\quad C_E\geq C^*,
\end{array}
\right.
    \end{equation}
with $\eta_1<0<\eta_2<x^*<\eta_3$ and $\overline{\eta}<0$. Notice that the periodic motions occurring for $0<C_E<C^*$ (Whose orbits are the closed curves). We see that
\begin{equation*}
    \dot{x_1}=-\frac{(a_1+a_2)A_1}{2\pi}\frac{x_2}{|x|^2}=\pm\frac{(a_1+a_2)A_1e^{-x_1/a'}}{2\pi C_E}\sqrt{C_Ee^{x_1/a'}-x_1^2},
\end{equation*}
then the period of a closed orbit on the level $C_E<C^*$ is given by
\begin{equation*}
    T_E=2\int^{\eta_2}_{\eta_1}\frac{dx_1}{|\dot{x_1}|}=\frac{4\pi C_E}{(a_1+a_2)A_1}\int^{\eta_2}_{\eta_1}\frac{e^{x_1/a'}}{\sqrt{C_Ee^{x_1/a'}-x_1^2}}dx_1,
\end{equation*}
with $\eta_1<0<\eta_2$ are the two smallest roots of the equation $C_Ee^{x_1/a'}-x_1^2=0$. Notice that for small positive $C_E$, we have $\eta_{1,2}\approx\mp\sqrt{C_E}$, and
\begin{equation}\label{34}
    T_E\approx\frac{4\pi C_E}{(a_1+a_2)A_1}\int^{\sqrt{C_E}}_{-\sqrt{C_E}}\frac{dx_1}{\sqrt{C_E-x_1^2}}=\frac{4\pi^2 C_E}{(a_1+a_2)A_1},
\end{equation}
which goes to $0$ as $C_E\to 0$(i.e., $E\to-\infty$, or say $x\to 0$).

If $x_0=0$, then $B_1=0$ and $\mathcal{H}(x)=\frac{a_1+a_2}{4\pi}A_1\ln |x|^2$, whose orbits are closed curves of any $E\in\mathbb{R}$, and for small $C_E=\exp\left(\frac{4\pi E}{(a_1+a_2)A_1}\right)$, the period of a close orbit still satisfies
\begin{equation*}
    T_E\approx\frac{4\pi^2 C_E}{(a_1+a_2)A_1}.
\end{equation*}

Now assume that $|P^0_1-P^0_2|=4\rho_0/n$, with $\rho_0>0$ and $n\in \mathbb{N}$, then $\rho\leq \rho_0/n$, assuming $\rho=\beta \rho_0/n$ with $\beta\in(0,1)$.
\begin{lemma}\label{35}
    Fix $\beta\in\left(0,\frac{h^2+h\sqrt{h^2+r_0^2}}{r_0^2+h^2+h\sqrt{h^2+r_0^2}}\right)$, then for any $k\in\mathbb{N}$, let $T=(k+1)T_E$, there exists $N_0>0$ such that if $n>N_0$ then $C_E<C^*$ and the solution $x_E(t)$ to  (\ref{31}) satisfies
    \begin{equation}\label{5-9}
        \min_{t\in[0,T]}|DT(x_0)^{-1}x_E(t)|=\min_{t\in[0,T]}|P_1(t)-P_2(t)|\geq 4\rho.
    \end{equation}
\end{lemma}
\begin{proof}
   The proof follows by standard arguments in the theory of ordinary differential equations, and combining (\ref{36}).
\end{proof}

Now we are ready to complete our proof of Theorem \ref{thm-leapfrogging}.
\begin{proof} Recalling that when $n$ is large enough, by (\ref{34}) we have $T_E\sim 1/n^2$. We follow the idea used in \cite{BCM3}  to finish our proof by an iterative argument in the sequel.

\textbf {Step} 0. We choose $R=\rho_n:=\rho/n$ and $l=3$ in Proposition \ref{23}, then we get $\overline{T}_3$, and in accordance with the discussion in Proposition \ref{27}, there exists $T_0'\in(0,T]$ such that
\begin{equation}\label{5-6}
    \Lambda_{i,\ep}(t)\subset B\left(B^{i,\ep}(t),\frac{3C_0\rho_n}{|\ln\ep|}\right)\quad \forall t\in [0,T_0'\wedge T_\ep].
\end{equation}
In this case,
\begin{equation*}
    t_1:=\inf\{t\in[0,\overline{T}_3\wedge T_\ep]:\tilde{R}_s\leq 3c_0C_0\rho_n\quad\forall s\in[0,t]\},
\end{equation*}
and(whenever $t_1<\overline{T}_3\wedge T_\ep$)
\begin{equation*}
    t_0:=\inf\{t\in[0,t_1]:\tilde{R}_s>c_0C_0\rho_n\quad \forall s\in[t,t_1]\}.
\end{equation*}
Now we choose $R(t_0)=2c_0\rho_n$. By the same argument as in Proposition \ref{27} (notice that $1\leq C_0$, we have that $\dot{R}(t)\leq3\rho^{-2}C_0\tilde{C}'+\rho^{-2}\tilde{C}'R(t)$), we deduce that   (\ref{5-6}) holds with
\begin{equation*}
    \begin{split}
        T_0'&=\frac{\rho^2}{\tilde{C}'}\ln\left(\frac{3\tilde{C}'c_0C_0\rho_n+3C_0\tilde{C}'\rho^2}{2\tilde{C}'c_0C_0\rho_n+3C_0\tilde{C}'\rho^2}\right)\wedge\overline{T}_3\\
        &=\frac{\rho^2}{\tilde{C}'}\ln\left(\frac{3\tilde{C}'c_0\beta\rho_0+3\tilde{C}'\beta^2\rho_0^2}{2\tilde{C}'c_0\beta\rho_0+3\tilde{C}'\beta^2\rho_0^2}\right)\wedge\overline{T}_3,
    \end{split}
\end{equation*}
where
\begin{equation*}
    \overline{T}_3=\frac{1}{\tilde{C}_A}\cdot\frac{C_0\rho^2\rho_n}{\rho^2+C_0\rho_n}e^{-5}\wedge T=\frac{\rho^2}{\tilde{C}_A}\cdot\frac{C_0\beta\rho_0}{\beta^2\rho_0^2+C_0\beta\rho_0}e^{-5}\wedge T.
\end{equation*}

\textbf {Step} 1. If $T_0'=T$ we are done, otherwise, from Step $0$ and   (\ref{29}), (\ref{28}) we have $T_\ep>T_0'$ for any $\ep$ small enough, and whence
\begin{equation*}
    \Lambda_{i,\ep}(T_0')\subset B\left(B^{i,\ep}(T_0'),\frac{3C_0\rho_n}{|\ln\ep|}\right).
\end{equation*}
Then we will repeat the arguments in Proposition \ref{20} and Proposition \ref{23} to deduce that, for any $\ep$ small enough, we have
\begin{equation}\label{5-7}
    m^i_t(4C_0^2\rho_n)\leq \varepsilon^l\quad\forall t\in[T_0',(T_0'+\overline{T}_l)\wedge T_\ep].
\end{equation}

We argue in the following way.

(i) First we follow the proof of Proposition \ref{20}, with $C_0^2(R-\rho_n/4)$ in place of $C_0R/2$ in the computation leading to   (\ref{17}), and iterate   (\ref{4-23}) from $C_0^2(3\rho_n+\rho_n/2-h)$ to $C_0^2(3\rho_n+\rho_n/4)$, in this case, $R=C_0^2(3\rho_n+\rho_n/2)$ and $h=C_0^2\rho_n/(4n+4)$, then we have
\begin{equation*}
    m^i_t(C_0^2(3\rho_n+\rho_n/2))\leq |\ln\ep|^{-4}\quad \forall t\in[0,(T_0'+\tilde{T}_{3-b})\wedge T_\ep],
\end{equation*}
where
\begin{equation*}
    \tilde{T}_2=\frac{\rho^2}{{C}_A}\cdot\frac{C_0^2(3\rho_n+\rho_n/2)}{C_0'\rho^2+C_0^2(3\rho_n+\rho_n/2)}e^{b-4}\wedge(T-T_0').
\end{equation*}

(ii) Second, by (i), following the proof of Proposition  \ref{23}  and iterate now from $C_0^2(4\rho_n-h)$ to $C_0^2(4\rho_n-\rho_n/4)$, then   (\ref{5-7}) holds for $l>2$, and in this case
\begin{equation*}
    \begin{split}
        \overline{T}_l&=\frac{\rho^2}{\tilde{C}_A}\cdot\frac{4C_0^2\rho_n}{\rho^2+4C_0^2\rho_n}e^{-(l+2)}\wedge(T-T_0')\\
        &=\frac{\rho^2}{\tilde{C}_A}\cdot\frac{4C_0^2\beta\rho_0}{\beta^2\rho_0^2+4C_0^2\beta\rho_0}e^{-(l+2)}\wedge(T-T_0').
    \end{split}
\end{equation*}
Using   (\ref{5-7}), we can now adjust the arguments in Section 4 to prove that, for any $\ep$ small enough, we have
\begin{equation}\label{5-8}
    \Lambda_{i,\ep}(t)\subset B\left(B^{i,\ep}(t),\frac{6C_0^3\rho_n}{|\ln\ep|}\right)\quad \forall t\in [T_0',(T_0'+T_1')\wedge T_\ep],
\end{equation}
where $T_1'\in(0,T-T_0']$ will be given a clear expression later.

We carry out our discussion as following.

(i) First we can modify the claim of Lemma \ref{25} by replacing   (\ref{22}) with
\begin{equation*}
    \dot{\tilde{R}}_t\leq\tilde{C}'\left(1+\frac{\tilde{R}_t}{\rho^2}+\frac{2}{|\ln\ep|^\sigma(R_t\wedge C_0^3\rho_n)^3}+\frac{|\ln\varepsilon|^b\sqrt{m^i_t(R^n_t)}}{\ep}\right),
\end{equation*}
where $R^n_t=(R_t-C_0^3\rho_n/2)\vee(R_t/2)$. It only requires us to change  the integral region of $H_1$ and $H_2$ to $B(B^{i,\ep}(t),R^n_t/|\ln\ep|)$ and $B(B^{i,\ep}(t),R_t/|\ln\ep|)\setminus B(B^{i,\ep}(t),R^n_t/|\ln\ep|)$.

(ii) Second letting $\overline{T}_3$ be as in  (\ref{5-7}) for $l=3$, we prove   (\ref{5-8}) following the proof of Proposition \ref{27}. In this case,
\begin{equation*}
    t_1:=\sup\{t\in[T_0',(T_0'+\overline{T}_3)\wedge T_\ep]:\tilde{R}_s\leq 6c_0C_0^3\rho_n\quad \forall s\in[0,t]\}.
\end{equation*}
and(whenever $t_1<(T_0'+\overline{T}_3)\wedge T_\ep$)
\begin{equation*}
    t_0:=\inf\{t\in[T_0',t_1]:\tilde{R}_s>c_0C_0^3(4\rho_n+\rho/2)\quad \forall s\in[t,t_1]\}.
\end{equation*}
Notice that when $t\in[t_0,t_1]$, we have $m^i_t(R^n_t)\leq m^i_t(4C_0^2\rho_n)\leq \ep^3$ by   (\ref{5-7}). Now choosing $R(t_0)=5C_0^3\rho_n$ in   (\ref{26}) we deduce that   (\ref{5-8}) holds with
\begin{equation*}
    T_1'=\frac{\rho^2}{\tilde{C}'}\ln\left(\frac{6\tilde{C}'c_0\beta\rho_0+3\tilde{C}'\beta^2\rho_0^2}{5\tilde{C}'c_0\beta\rho_0+3\tilde{C}'\beta^2\rho_0^2}\right)\wedge \overline{T}_3,
\end{equation*}
where
\begin{equation*}
    \overline{T}_3=\frac{\rho^2}{\tilde{C}_A}\cdot\frac{4C_0^2\beta\rho_0}{\beta^2\rho_0^2+4C_0^2\beta\rho_0}e^{-5}\wedge(T-T_0').
\end{equation*}

\textbf {Step} 2. If $T_0'+T_1'=T$ we are done, otherwise from   (\ref{5-8}), (\ref{29}) and (\ref{28}) we have $T_\ep>T_0'+T_1'$ for any $\ep$ small enough, and whence
\begin{equation*}
    \Lambda_{i,\ep}(T_0'+T_1')\subset B\left(B^{i,\ep}(T_0'+T_1'),\frac{6C_0^3\rho_n}{|\ln\ep|}\right).
\end{equation*}
Therefore, through a discussion similar to that in Step $1$, we have that, for any $\ep$ small enough,
\begin{equation*}
    m^i_t(7C_0^4\rho_n)\leq \ep^l\quad \forall t\in[T_0'+T_1',(T_0'+T_1'+\overline{T}_l)\wedge T_\ep],
\end{equation*}
with
\begin{equation*}
    \overline{T}_l=\frac{\rho^2}{\tilde{C}_A}\cdot \frac{7C_0^4\beta\rho_0}{\beta^2\rho_0^2+7C_0^4\beta\rho_0}e^{-(l+2)}\wedge (T-T_0'-T_1'),
\end{equation*}
whence
\begin{equation*}
    \Lambda_{i,\ep}(t)\subset B\left(B^{i,\ep}(t),\frac{9C_0^5\rho_n}{|\ln\ep|}\right) \quad \forall t\in[T_0'+T_1',(T_0'+T_1'+T_2')\wedge T_\ep],
\end{equation*}
with
\begin{equation*}
    T_2'=\frac{\rho^2}{\tilde{C}'}\ln\left(\frac{9\tilde{C}'c_0\beta\rho_0+3\tilde{C}'\beta^2\rho_0^2}{8\tilde{C}'c_0\beta\rho_0+3\tilde{C}'\beta^2\rho_0^2}\right)\wedge \overline{T}_3,
\end{equation*}
and
\begin{equation*}
    \overline{T}_3=\frac{\rho^2}{\tilde{C}_A}\cdot\frac{7C_0^4\beta\rho_0}{\beta^2\rho_0^2+7C_0^4\beta\rho_0}e^{-5}\wedge(T-T_0'-T_1').
\end{equation*}

\textbf {Step $j$}. The above procedure can be iterated inductively in the following manner. If at the $(j-1)$th step we still have $T_0'+\cdots+T_{j+1}'<T$ (otherwise we are done) and $3jC_0^{2j-1}\rho_n\leq\rho/2$, then $T_\ep>T_0'+\cdots+T_{j-1}'$ for any $\ep$ small enough, so that
\begin{equation*}
    \Lambda_{i,\ep}(T_0'+\cdots+T_{j-1}')\subset B\left(B^{i,\ep}(T_0'+\cdots+T_{j-1}'),\frac{3jC_0^{2j-1}\rho_n}{|\ln\ep|}\right).
\end{equation*}
Now we can move on to the next iteration, giving first
\begin{equation*}
    m^i_t((3j+1)C_0^{2j}\rho_n)\leq \ep^l\quad \forall t\in[T_0'+\cdots+T_{j-1}',(T_0'+\cdots+T_{j-1}'+\overline{T}_l)\wedge T_\ep],
\end{equation*}
with
\begin{equation*}
    \overline{T}_l=\frac{\rho^2}{\tilde{C}_A}\cdot \frac{(3j+1)C_0^{2j}\beta\rho_0}{\beta^2\rho_0^2+(3j+1)C_0^{2j}\beta\rho_0}e^{-(l+2)}\wedge (T-(T_0'+\cdots+T_{j-1}')),
\end{equation*}
whence
\begin{equation*}
    \Lambda_{i,\ep}(t)\subset B\left(B^{i,\ep}(t),\frac{(3j+3)C_0^{2j+1}\rho_n}{|\ln\ep|}\right) \quad \forall t\in[T_0'+\cdots+T_{j-1}',(T_0'+\cdots+T_{j-1}')\wedge T_\ep],
\end{equation*}
with
\begin{equation*}
    T_j'=\frac{\rho^2}{\tilde{C}'}\ln\left(\frac{(3j+3)\tilde{C}'c_0\beta\rho_0+3\tilde{C}'\beta^2\rho_0^2}{(3j+2)\tilde{C}'c_0\beta\rho_0+3\tilde{C}'\beta^2\rho_0^2}\right)\wedge \overline{T}_3,
\end{equation*}
and
\begin{equation*}
    \overline{T}_3=\frac{\rho^2}{\tilde{C}_A}\cdot\frac{(3j+1)C_0^{2j}\beta\rho_0}{\beta^2\rho_0^2+(3j+1)C_0^{2j}\beta\rho_0}e^{-5}\wedge(T-(T_0'+\cdots+T_{j-1}')).
\end{equation*}
In conclusion, the upper limit on the number of iterations is specified by $j_*=j_n\wedge j_T$, where
\begin{equation*}
    j_n=\max\{j:3(j+1)C_0^{2j+1}\rho_n\leq\rho\}=O(\ln n),
\end{equation*}
\begin{equation*}
    j_T=\max\{0<j\leq j_n:T_0'+\cdots+T_{j-1}'<T\}.
\end{equation*}
From the explicit expression of $T_j'$, if $j<j_T$ then
\begin{equation*}
    T_j'=\frac{\rho^2}{\tilde{C}'}\ln\left(\frac{(3j+3)\tilde{C}'c_0\beta\rho_0+3\tilde{C}'\beta^2\rho_0^2}{(3j+2)\tilde{C}'c_0\beta\rho_0+3\tilde{C}'\beta^2\rho_0^2}\right)\wedge \frac{\rho^2}{\tilde{C}_A}\cdot\frac{(3j+1)C_0^{2j}\beta\rho_0}{\beta^2\rho_0^2+(3j+1)C_0^{2j}\beta\rho_0}e^{-5}.
\end{equation*}
Since $T_j'=O(j^{-1})\rho^2$, whence $\sum^{j_n}_{j=0}T_j'=O(\ln j_n)\rho^2=O(\ln \ln n)\rho^2$, recall that $\rho=\beta\rho_0/n$, then by choosing $n$ large enough, we have $\sum^{j_n}_{j=0}T_j'>T$, which means
\begin{equation*}
    j_*<j_n\quad and \quad\sum^{j_*}_{j=0}T_j'>T.
\end{equation*}
 This implies that $T_\rho'$ can cover $kT_E$ for any $k\in\mathbb{N}$, if we choose $n$ large enough, i.e., initial distance $|\tilde{P}^0_1-\tilde{P}^0_2|=\rho_0/n$ small enough.  Then the conclusion of Theorem \ref{thm-leapfrogging} follows immediately from the above discussion.
\end{proof}
\begin{remark}
    Notice that for small positive $C_E$, the actual orbit of $P_1(t)-P_2(t)$ is approximately elliptical, which is different from the trajectory which approximates a circle in \cite{BCM3}. So to ensure that (\ref{5-9}) holds true and then extends $T$ in (\ref{defT}), it is necessary to require that $\beta\in(0,c_0/C_0)$. Let us consider a special case: $(P_1^0)_2=(P_2^0)_2$, then the second entry of $P_1^0-P_2^0$ vanishes and we have $|x_E(0)|=C_0|P_1^0-P_2^0|$, and $\min_{t\in[0,T]}|P_1(t)-P_2(t)|\geq \min_{t\in[0,T]}|x_E(t)|/C_0\geq4\rho=4\beta\rho_0/n$ for any $\beta\in(0,1)$ if $n$ is large enough. That is, the constraint that $\beta \in(0,c_0/C_0)$ in Theorem \ref{thm-leapfrogging} and Lemma \ref{35} can be extended to $\beta\in(0,1)$.
\end{remark}

 \subsection*{Acknowledgments:}

 \par
 D. Cao and J. Fan were supported by  National Key R\&D Program of China (Grant 2022YFA1005602) and NNSF of China (Grant No. 12371212). G. Qin was supported by NNSF of China (Grant No.  12471190). J. Wan was supported by NNSF of China (Grants No. 12101045, 12471190).

 \subsection*{Conflict of interest statement} On behalf of all authors, the corresponding author states that there is no conflict of interest.

 \subsection*{Data availability statement} All data generated or analysed during this study are included in this published article  and its supplementary information files.

 \phantom{s}
 \thispagestyle{empty}


\end{document}